\documentclass[11pt,leqno]{article}

\usepackage[margin=1in]{geometry}
\usepackage{amsmath,amssymb,amsfonts,amsthm,amsopn}
\usepackage{bm,mathtools}
\usepackage{stmaryrd}
\usepackage{subcaption}
\usepackage{tikz}
\usepackage{booktabs}
\usepackage{mathrsfs}
\usepackage{graphicx}
\usepackage{ifpdf}
\usepackage[hidelinks]{hyperref}

\DeclareGraphicsExtensions{.pdf,.png,.jpg}
\numberwithin{equation}{section}

\newenvironment{keywords}
  {\par\smallskip\noindent\textbf{Keywords: }}
  {\par\smallskip}
\newenvironment{AMS}
  {\par\noindent\textbf{AMS subject classifications: }}
  {\par\smallskip}

\newcommand{\funding}[1]{\par\textbf{Funding:} #1}
\newcommand{\email}[1]{\href{mailto:#1}{#1}}

\newtheorem{assumption}{Assumption}

\usetikzlibrary{calc,arrows.meta,shapes.geometric,positioning}

\tikzset{
  interface/.style={line width=1.05pt},
  interfaceplane/.style={fill=gray!25,draw=black,line width=0.8pt,fill opacity=0.72},
  cubeedge/.style={line width=0.65pt},
  hiddenedge/.style={line width=0.55pt,dashed},
  positive/.style={circle,fill=black,draw=black,inner sep=1.8pt},
  negative/.style={circle,fill=white,draw=black,line width=0.7pt,inner sep=1.65pt},
  mpoint/.style={star,star points=5,star point ratio=2.2,fill=white,draw=black,line width=0.75pt,inner sep=2.0pt},
  axis/.style={-{Latex[length=2mm]},thin},
  maparrow/.style={-{Latex[length=3mm]},line width=0.8pt}
}

\newcommand{\Pos}[1]{\node[positive] at #1 {};}
\newcommand{\Neg}[1]{\node[negative] at #1 {};}

\newcommand{\CubeMinusOneOne}{%
  \coordinate (A) at (-1,-1,-1);
  \coordinate (B) at ( 1,-1,-1);
  \coordinate (C) at ( 1, 1,-1);
  \coordinate (D) at (-1, 1,-1);
  \coordinate (E) at (-1,-1, 1);
  \coordinate (F) at ( 1,-1, 1);
  \coordinate (G) at ( 1, 1, 1);
  \coordinate (H) at (-1, 1, 1);
  \draw[hiddenedge] (A)--(D)--(H);
  \draw[hiddenedge] (D)--(C);
  \draw[cubeedge] (A)--(B)--(C)--(G)--(F)--(B);
  \draw[cubeedge] (E)--(F);
  \draw[cubeedge] (E)--(H)--(G);
  \draw[cubeedge] (A)--(E);
}

\usetikzlibrary{calc,intersections}
\usetikzlibrary{calc,angles,quotes}

\theoremstyle{plain}
\newtheorem{theorem}{Theorem}[section]
\newtheorem{lemma}[theorem]{Lemma}

\theoremstyle{remark}
\newtheorem{remark}[theorem]{Remark}

\title{Quadrilateral and Hexahedral Immersed Finite Element Methods for Elliptic Interface Problems\thanks{Submitted to the editors DATE.\funding{This work was funded by NSFC under Grant No. 12371370.}
}}

\author{Fangfang Qin\thanks{Jiangsu Key Laboratory of Quantum Computing Science and Devices, School of Science, Nanjing University of Posts and Telecommunications, Nanjing, Jiangsu 210023, China.}
\and Haifeng Ji\footnotemark[2]~$^,$\thanks{Corresponding author: \email{hfji1988@foxmail.com}}
}
\date{}

\ifpdf
\hypersetup{
  pdftitle={Quadrilateral and Hexahedral Immersed Finite Element Methods for Elliptic Interface Problems},
  pdfauthor={F. Qin and H. Ji}
}
\fi

\begin{document}

\maketitle

\begin{abstract}
Immersed finite element (IFE) methods provide an effective framework for solving interface problems on unfitted meshes.
The basic idea underlying IFE methods is to modify standard finite element spaces by enforcing interface conditions at certain points.
It is known that, for nodal degrees of freedom, unisolvence of linear IFE basis functions on triangular elements is subject to a non-obtuse-angle condition for scalar diffusion coefficients. 
This limitation is more severe for tensor diffusion coefficients, for which the non-obtuse-angle condition is no longer sufficient for unisolvence. 
Even on rectangular meshes, conventional bilinear IFE basis functions may not be unisolvent for tensor diffusion coefficients.
In this paper, we develop and analyze a nodal isoparametric IFE method on quadrilateral and hexahedral meshes and show that the unisolvence issue can be overcome by appropriately selecting the enforcement point of the discrete flux condition. 
The key observation is that, unlike triangular elements, the discrete flux in quadrilateral and hexahedral elements is not constant, which provides the flexibility to select such an enforcement point to ensure unisolvence.
We provide a systematic procedure for  this selection  that not only ensures unisolvence of the IFE basis functions 
on general quadrilateral and hexahedral elements with either scalar or tensor-valued diffusion coefficients but also preserves the optimal approximation properties of the resulting IFE space. 
The proposed  IFE method offers several advantages: flexibility for complex geometries, the absence of angle restrictions, and applicability to tensor diffusion coefficients, thereby overcoming the limitations of existing 
nodal IFE methods on rectangular and triangular meshes.
Optimal error estimates are established and confirmed by numerical experiments.
\end{abstract}



\section{Introduction}
Interface problems arise in a wide range of applications, including heat transfer, multiphase flow, and materials science, where 
the computational domain is partitioned into subdomains occupied by different materials and the physical parameters may be discontinuous across the interfaces separating these subdomains.
In addition to the governing equations in each subdomain, appropriate interface conditions must be satisfied according to physical laws, which may lead to jumps in the solution or its derivatives across the interface. 
Such reduced regularity must be carefully taken into account in the design of accurate numerical methods.

Let $\Omega\subset\mathbb{R}^N$, $N\in\{2,3\}$, be a bounded convex polygonal/polyhedral domain, and let $\Gamma\subset \Omega$ be a closed $C^2$ hypersurface that separates $\Omega$ into two subdomains \(\Omega_\circ^\pm\).
For notational convenience in the subsequent
analysis, we assign \(\Gamma\) to the positive side by setting
\(
\Omega^+ := \Omega_\circ^+\cup\Gamma
\)
and
\(
\Omega^- := \Omega_\circ^-.
\)
For any function $v$ and region $T$, define $v^\pm=v|_{\Omega^\pm}$ and $T^\pm =T\cap \Omega^\pm$. 
We consider the  model problem 
\begin{subequations}\label{p1}
\begin{align}
-\nabla\cdot(\mathbb{B}^\pm \nabla u^\pm)&=f^\pm
&&\quad \text{in } \Omega^\pm, \label{p1.1}\\
\llbracket u^\pm  \rrbracket&:=u^+-u^-=0 
&&\quad \text{on } \Gamma, \label{p1.2}\\
\llbracket \mathbb{B}^\pm \nabla u^\pm\cdot n \rrbracket&:=\mathbb{B}^+ \nabla u^+\cdot n-\mathbb{B}^- \nabla u^-\cdot n =0 
&&\quad \text{on } \Gamma, \label{p1.3}\\
u&=0 
&&\quad \text{on } \partial\Omega. \label{p1.4}
\end{align}
\end{subequations}
Here $f\in L^2(\Omega)$, $n$ denotes the unit normal to $\Gamma$ pointing toward $\Omega^+$, and  $\mathbb{B}(x)=(b_{ij}(x))_{i,j=1}^N$ is a symmetric positive definite matrix-valued function whose components are piecewise smooth and may be discontinuous across the interface $\Gamma$. 
We assume that $b_{ij}^{\pm}\in C^1(\overline{\Omega^\pm})$ and that there exist positive constants $\beta_{\min}^\pm$ and $\beta_{\max}^\pm$ such that
\begin{equation}\label{def_betaM}
\beta_{\min}^\pm y^\top y
\leq
y^\top\mathbb{B}^{\pm}(x)y
\leq
\beta_{\max}^\pm y^\top y,
\qquad 
\forall y\in\mathbb{R}^N,\quad \forall x\in\Omega^\pm .
\end{equation}

A natural approach to solving interface problems is to employ interface-fitted meshes, in which element boundaries are aligned with the exact interface or with a prescribed geometric approximation of it \cite{bramble1996finite}.
However, generating and updating high-quality interface-fitted meshes can be challenging and computationally expensive, especially for complex or evolving interfaces in three dimensions. 
These difficulties have motivated the development of unfitted mesh methods, in which the mesh is generated independently of the interface.

Among the various unfitted mesh methods, immersed finite element (IFE) methods have received considerable attention due to their simplicity and compatibility with standard finite element frameworks. 
The basic idea of IFE methods is to incorporate interface conditions into finite element spaces.
In contrast to enrichment-based methods (see, e.g., \cite{hansbo2002unfitted,fries2010extended,burman2018cut,Wang2018A}), the IFE space retains the same degrees of freedom as the corresponding standard finite element space. Consequently, IFE methods preserve the conventional finite element structure and can be naturally integrated into existing finite element codes.
The first IFE method was introduced in \cite{li1998immersed} to solve one-dimensional interface problems and was recently revisited in \cite{wang2021robust}. 
Over the past several decades, IFE methods have been extensively analyzed and extended to various interface problems; see, e.g., \cite{Li2003new,he2012convergence,kafafy2005three,taolin2015siam,GuzmanJSC2017,guo2020immersed,2021ji_IFE,ji2023immersed}.

\subsection{Motivations}
The unisolvence of local IFE basis functions, namely, the unique determination of functions in the local IFE space by the prescribed degrees of freedom, is a fundamental issue in the construction of IFE spaces.
For the conventional nodal linear IFE construction with piecewise constant scalar diffusion coefficients, the non-obtuse-angle
condition guarantees unisolvence for arbitrary interface locations and positive coefficient pairs \cite{2021ji_IFE}, whereas
unisolvence may fail on obtuse triangles.
%
A counterexample was also provided in \cite{2021ji_IFE} showing that unisolvence fails for an obtuse triangle.
Furthermore, for tensor diffusion coefficients, unisolvence may fail even when the non-obtuse-angle condition is satisfied. Counterexamples were provided in \cite{An2014A}.

\begin{figure}[ht]
\centering
\begin{tikzpicture}[scale=1.12,line cap=round,line join=round]
\begin{scope}[xshift=0cm]
  \coordinate (A1) at (0,0);
    \coordinate (D) at (0,0);
    \coordinate (A2) at (1.5,0);
  \coordinate (A3) at (-2.598,1.5);
  \coordinate (E)  at (-0.402,0.696);
  
  \coordinate (A2_ext) at ($(A2)!-0.3!(A3)$);

\draw[dashed,thick] (A3)--(A2_ext);

  \draw[thick] (A1)--(A2)--(A3)--cycle;
  \draw[very thick] (A1)--(E);

  \fill (A1) circle (1.6pt);
  \fill (A2) circle (1.6pt);
  \fill (A3) circle (1.6pt);
  \fill (E)  circle (1.6pt);

  \node[below ] at (A1) {$D=A_1$};
  \node[above ] at (A3) {$A_3$};
  \node[above] at (A2) {$A_2$};
  \node[above =1pt] at (E) {$E$};

  \node at (-1.7,0.6) {$T^-$};
  \node at (0.8,0.7) {$T^+$};
  
\coordinate (Eta) at ($(E)!0.8!(D)$);

\coordinate (Xi) at ($(E)!0.8!90:(D)$);

\draw[->,very thick,] (E)--(Xi)
    node[right, above] {$n_h$};

\draw[->,very thick] (E)--(Eta)
    node[right, yshift=3pt ] {$t_h$};

\coordinate (Eta1) at ($(E)!-1.0!(D)$);
\coordinate (Eta2) at ($(E)!1.0!(D)$);

\coordinate (Xi1) at ($(E)!-2.0!90:(D)$);
\coordinate (Xi2) at ($(E)!2.0!90:(D)$);

\draw[dashed,thick] (Eta1)--(Eta2);
\draw[dashed,thick] (Xi1)--(Xi2);

\path[name path=normalLine] (Xi1)--(Xi2);
\path[name path=edgeLine] (A3)--(D);
\path[name intersections={of=normalLine and edgeLine, by=F}];
\fill (F) circle (1.6pt);
\node[below=3pt] at (F) {$F$};

\coordinate (XiBack) at ($(E)!-0.8!90:(D)$);

\pic [draw, thick, angle radius=5pt] 
{right angle = XiBack--E--Eta};

\pic [
    draw,
    thick,
    angle radius=8pt,
    "$\theta$",
    angle eccentricity=1.5
] {angle = A2--E--Xi};

\end{scope}
\end{tikzpicture}
\caption{A geometric configuration illustrating the loss of unisolvence for nodal linear IFE basis functions on an obtuse interface triangle.}
\label{fig:unisolvence-motivation}
\end{figure}
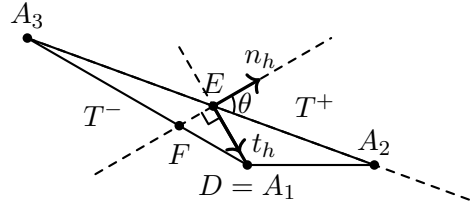

To understand the underlying reason for the loss of unisolvence for obtuse triangles and scalar diffusion coefficients, consider the obtuse interface triangle $T=\triangle A_1A_2A_3$ shown in Fig.~\ref{fig:unisolvence-motivation}, for which $\Gamma\cap\partial T=\{D,E\}$ and $D=A_1$. 
A local IFE function $\phi$ is defined by
\[
\phi|_{T^\pm}=\phi^\pm,\qquad
\phi^\pm(x_1,x_2)=a^\pm x_1+b^\pm x_2+c^\pm,
\]
which is intended to be uniquely determined by the three nodal values
$\phi(A_i)$, $i=1,2,3$, together with the discrete interface conditions
\begin{equation}\label{dis_int_con}
\phi^+(D)=\phi^-(D),\qquad
\phi^+(E)=\phi^-(E),\qquad
\beta^+\nabla\phi^+\cdot n_{h}
=
\beta^-\nabla\phi^-\cdot n_{h},
\end{equation}
where $n_h$ is the unit normal to $\overline{DE}$ pointing toward $T^+$.
The loss of unisolvence occurs if there exists  a nontrivial function $\phi\neq 0$ satisfying \eqref{dis_int_con} and $\phi(A_1)=\phi(A_2)=\phi(A_3)=0$. 
Set $\phi^+(E)=\phi^-(E)=\phi(E)=1$ and $\phi^+(D)=\phi^-(D)=\phi(A_1)=0$. Then the linear functions $\phi^\pm\neq 0$ are uniquely determined. 
It remains to choose positive coefficients $\beta^\pm$ so that the flux continuity condition in \eqref{dis_int_con} is satisfied.
It suffices to prove that 
$
k^\pm:=\nabla \phi^\pm\cdot n_h>0.
$
This is because choosing positive coefficients $\beta^\pm$ satisfying
$
\beta^+/\beta^-=k^-/k^+
$
guarantees the flux continuity condition in \eqref{dis_int_con}.

The signs of \(k^\pm\) can be verified geometrically. 
Let \(F\) be the intersection of \(\overline{A_1A_3}\) with the line
through \(E\) parallel to \( n_h\), and let
\(\theta \) be the acute angle between
\( n_h\) and \(\overrightarrow{EA_2}\), as shown in
Fig.~\ref{fig:unisolvence-motivation}.
A direct calculation gives
\[
k^-=\frac{1}{|EF|}>0,
\qquad
k^+=\frac{\tan\theta}{|DE|}
      -\frac{\sec\theta}{|EA_2|}.
\]
Thus, \(k^+>0\) can be ensured by moving \(A_2\) sufficiently far from
\(E\) while keeping the direction of \(A_3A_2\) fixed.

%


Note that in \cite{ji2023immersed}, the non-obtuse-angle restriction can be removed by employing edge- and face-average degrees of freedom.
This naturally raises the question of whether the angle restriction can also be removed while retaining nodal degrees of freedom.
This question forms the first motivation of the present work.

The second motivation is to extend IFE constructions beyond rectangular and cubical meshes. Many existing IFE methods have been developed on such meshes, which provide simple mesh structures but may offer limited flexibility in representing complex computational domains.
Unlike IFE constructions on rectangular and cubical elements, where IFE functions can be constructed directly on physical elements, general quadrilateral and hexahedral elements require isoparametric mappings from reference elements to physical elements, which are generally non-affine.
Under such mappings, the normal direction of the interface is not preserved, and the transformation of the normal flux involves the Jacobian matrix of the mapping.
Consequently, the flux interface condition becomes more complicated after transforming to reference elements. 
Therefore, constructing and analyzing IFE spaces on general quadrilateral and hexahedral meshes is more challenging.

\subsection{Contributions}
In this work, we develop and analyze isoparametric IFE methods that employ nodal degrees of freedom for problem \eqref{p1} on general unfitted quadrilateral and hexahedral meshes. 
Using isoparametric mappings, IFE spaces are constructed on reference elements, where the interface and the corresponding interface conditions are transformed accordingly. 
We show that, for scalar diffusion coefficients, the standard construction of IFE spaces also requires a non-obtuse-angle condition, which restricts the admissible quadrilateral elements to rectangles. 
A counterexample is constructed on a mildly skewed parallelogram to demonstrate the loss of unisolvence. 
Furthermore, for tensor diffusion coefficients, we provide another counterexample showing that unisolvence may fail on rectangular elements. Both counterexamples are presented in Appendix~\ref{appendix}.

To address the unisolvence issue, we observe that the discrete normal flux on quadrilateral and hexahedral elements is generally nonconstant. This provides additional freedom in choosing the point at which the discrete flux-continuity condition is imposed.
Interestingly, we find that an appropriate choice of the flux-enforcement point does indeed ensure the unisolvence of the local IFE basis functions on general quadrilateral and hexahedral elements with either scalar or tensor-valued diffusion coefficients.
We note that this strategy is not applicable to triangular IFEs, since the discrete flux is constant on each side of the interface and changing the flux-enforcement point does not affect the flux continuity condition in \eqref{dis_int_con}.

For both quadrilateral and hexahedral elements, we identify the
flux-enforcement point using the same general strategy. The
three-dimensional construction, however, requires a more elaborate
algebraic treatment because of the increased complexity of the
interface configurations.
We establish the optimal approximation properties of the resulting IFE spaces, although the flux continuity condition is no longer enforced exactly on the interface. The underlying reason is that the flux-enforcement point can be chosen within the element and is therefore only an $O(h)$ perturbation of a point on the exact interface, which does not affect the optimal approximation properties.
We also establish optimal error estimates for the proposed IFE methods, with constants independent of the interface position relative to the mesh.

Compared with existing nodal IFE methods on rectangular and cubical meshes, the proposed method accommodates general quadrilateral and hexahedral meshes, requires no restrictive angle conditions, and applies to tensor-valued diffusion coefficients. It therefore provides greater flexibility for meshing complex computational domains. To the best of our knowledge, this is the first nodal IFE construction whose unisolvence is guaranteed for every admissible interface configuration and every pair of symmetric positive definite diffusion tensors on both quadrilateral and hexahedral elements.

The remainder of this paper is organized as follows.
In Section~\ref{sec_construction}, we construct IFE spaces and derive explicit formulas for IFE basis functions. In Section~\ref{sec_flux_point_selection}, we identify the flux-enforcement points.
Section~\ref{sec_approx} establishes the approximation properties of the IFE spaces. In Section~\ref{sec_method}, we formulate the IFE methods and derive error estimates. Numerical results are presented in Section~\ref{sec_num}.

\section{Construction of IFE spaces}\label{sec_construction}
Let $\{\mathcal{K}_h\}_{h>0}$ be a family of conforming quadrilateral meshes for $N = 2$ and conforming hexahedral meshes for $N = 3$, generated independently of the interface $\Gamma$. The diameter of $K \in\mathcal{K}_h$ is $h_K$ and $h = \max_{K\in \mathcal{K}_h} h_K$.
An element \(K\in\mathcal{K}_h\) is called an interface element if
\(
    \Gamma\cap\operatorname{int}(K)\neq\emptyset;
\)
otherwise, it is called a non-interface element.
The sets of interface and non-interface elements are denoted by
\(\mathcal{K}_h^\Gamma\) and \(\mathcal{K}_h^{\rm non}\), respectively.

We take \( \widehat K =[-1,1]^N\) as the reference element and define the multilinear polynomial space
\[
Q_1(\widehat K)
=
\operatorname{span}
\left\{
\widehat{x}_1^{\alpha_1}\cdots
\widehat{x}_N^{\alpha_N}:
\alpha_i\in\{0,1\},\ i=1,\ldots,N
\right\}.
\]
The  physical element \(K\in\mathcal{K}_h\) is represented as the image of
\( \widehat K \)
under a bijective mapping
\(
    F_K \in Q_1(\widehat K)^N,
\)
which is assumed to be orientation-preserving, namely,  the Jacobian matrix  \(J_K:=DF_K\)  of $F_K$ satisfies 
\(
    \det J_K(\widehat x)>0
 \)
 for all
 \(
\widehat x\in\widehat K.
\)
The standard isoparametric finite element space on  $K$ is 
\[
   V(K)
    =
    \left\{
        v:
        v=\widehat v\circ F_K^{-1},
        \quad
        \widehat v\in Q_1(\widehat K)
    \right\}.
\]

On each interface element $K\in\mathcal K_h^\Gamma$, we modify $V(K)$ to incorporate the interface conditions \eqref{p1.2}--\eqref{p1.3}.  
We pull back the interface geometry and the interface conditions to the reference element and perform the modification there.

In the following subsections, we fix an arbitrary interface
element \(K\in\mathcal K_h^\Gamma\) and suppress the dependence on \(K\)
in the notation whenever no confusion can arise.
\subsection{Interface conditions on the reference element}
The pullback of a scalar-valued  function $v$ on $K$ is defined by
\(
    \widehat v=v\circ F_K,
\)
and the pushforward of $\widehat v$ onto $K$ is defined by
\(
    v=\widehat v\circ F_K^{-1}.
\)
The chain rule implies 
\[
    \nabla v(x)
    =
    J_K(\widehat x)^{-\top}
   \widehat \nabla \widehat v(\widehat x),
    \qquad x=F_K(\widehat x),
\]
where \(\widehat\nabla\) denotes the gradient with respect to \(\widehat x\).

Let \(\Gamma_K=\Gamma\cap K\) and 
\(
    \widehat\Gamma_K=F_K^{-1}(\Gamma_K).
\)
Let $\widehat n$  denote the  unit normal to $\widehat\Gamma_K$  pointing toward $\widehat K^+:=F_K^{-1}(K^+)$. 
Since $n$ is the unit normal vector to $\Gamma_K$ pointing toward $K^+$, the transformation $x=F_K(\widehat x)$ gives
\[
 n(x) = \frac{J_K^{-\top}\widehat n}
    {|J_K^{-\top}\widehat n|}(\widehat x).
\]
Thus, for the exact solution, we have
\[
(\mathbb B^\pm
    \nabla u^\pm\cdot n)(x)  =    \frac{
        J_K^{-1}(\widehat x)  \mathbb B^\pm(F_K(\widehat x))
       J_K^{-\top}(\widehat x)
    \widehat\nabla\widehat  u^\pm(\widehat x)\cdot   \widehat n(\widehat x)
    }{
        |J_K^{-\top}(\widehat x)\widehat n(\widehat x)|
    }.
\]
Define the transformed diffusion tensor by
\begin{equation}\label{def_b_ref}
    \widehat{\mathcal B}_K^\pm
    =
 J_K^{-1}  \widehat{\mathbb B}^\pm
       J_K^{-\top},\quad  \widehat{\mathbb B}^\pm ={\mathbb B}^\pm \circ F_K.
\end{equation}
Then, by  \eqref{p1.2}--\eqref{p1.3} and the fact that
\(
    |J_K^{-\top}\widehat n|
\)
has no jump across \(\widehat\Gamma_K\),  we obtain
\begin{equation}\label{jump_con_ref}
\llbracket \widehat u^\pm  \rrbracket|_{\widehat\Gamma_K}=0,\quad
    \llbracket
        \widehat{\mathcal B}_K^\pm
        \widehat\nabla\widehat u^\pm\cdot\widehat n
    \rrbracket|_{\widehat\Gamma_K}
    =0.
\end{equation}

\subsection{Local IFE spaces}
We first construct some approximations of 
\(\Gamma_K\) and \(\widehat{\Gamma}_K\).
We  select a point \(x_K\in K\)   and a constant unit vector \(n_{h}\) on the physical element \(K\).
On the reference element,  we define \(\widehat x_K\) and \(\widehat n_{h}\) as
\[
\widehat x_K=F_K^{-1}(x_K),\qquad
\widehat n_h=
\frac{J_K(\widehat x_K)^\top n_h}
{|J_K(\widehat x_K)^\top n_h|},
\]
and define a hyperplane by
\begin{equation}\label{def_H_ref}
    \widehat{\mathcal H}_K
    :=
    \widehat{\mathcal H}(\widehat x_K,\widehat n_{h})
    =
    \left\{
        \widehat x\in \mathbb{R}^N\,:\,
        \widehat \ell_K(\widehat x)=0
    \right\},\quad \widehat \ell_K(\widehat x)=(\widehat x-\widehat x_K)\cdot\widehat n_{h}.
\end{equation}
The planar approximation of \(\widehat{\Gamma}_K\) on the reference
element is then defined as
\(
\widehat{\Gamma}_{h,K}=\widehat{\mathcal H}_K\cap \widehat{K}.
\)
On the physical element \(K\), the interface \(\Gamma_{K}\) is approximated by
\(
\Gamma_{h,K}=F_K(\widehat{\Gamma}_{h,K}),
\)
which is not necessarily a  hyperplane patch since \(F_K\) is a non-affine mapping.

\begin{assumption}
\label{ass:compatible_cut}
The point \(x_K\in K\)   and the vector \(n_{h}\) are chosen such that 
\begin{equation}\label{est_x_K}
    |x_K-x_{\Gamma_K}|+h_K|n_{h}-n(x_{\Gamma_K})|\leq C h_K^2
\end{equation}
for some point \(x_{\Gamma_K}\in\Gamma_K\). 
Moreover, we assume that the pulled-back exact interface
\(\widehat\Gamma_K\) intersects each edge of \(\widehat K\) at most once and that the exact and discrete interfaces induce the same classification of the vertices of \(\widehat K\).
More precisely, let \(\widehat {\mathcal N}\) denote the set of  vertices of
\(\widehat K\), and define
\begin{equation}\label{def_I_K_class}
\begin{aligned}
&\mathcal I_K^+
=
\bigl\{
 \widehat A_i :\,  \widehat A_i\in \widehat K^+,\, \widehat A_i\in \widehat {\mathcal N}
\bigr\},\quad &&\mathcal I_K^-=\widehat {\mathcal N} \backslash \mathcal I_K^+,\\
&\mathcal I_{h,K}^+
=
\bigl\{
\widehat A_i : \, 
\widehat \ell_K (\widehat A_i)\geq0,\, \widehat A_i\in\widehat {\mathcal N}
\bigr\},\quad &&\mathcal I_{h,K}^-=\widehat {\mathcal N} \backslash \mathcal I_{h,K}^+.
\end{aligned}
\end{equation}
We assume that
\(
    \mathcal I_K^s=\mathcal I_{h,K}^s\not= \emptyset
\)
for
\(  s\in\{+, -\}.
\)
\end{assumption}


%
To handle variable diffusion coefficients, we choose arbitrary points
$\widehat {x}_{K^\pm}\in \widehat K^\pm=F_K^{-1}(K^\pm)$
and define the frozen diffusion matrices by
\begin{equation}\label{frozen_B}
 \widehat{\mathbb {B}}^\pm_{h}=\widehat{\mathbb {B}}^\pm(\widehat x_{K^\pm})={\mathbb {B}}^\pm(x_{K^\pm}),\quad x_{K^\pm}=F_K(\widehat x_{K^\pm}).
\end{equation}
Accordingly, by \eqref{def_b_ref}, we define 
\begin{equation}\label{B_h_trans}
    \widehat{\mathcal B}_{h}^\pm(\widehat x)
    =
 J_K^{-1}(\widehat x) \widehat {\mathbb B}^\pm_{h}
       J_K^{-\top}(\widehat x).
\end{equation}
Let \(\bm \alpha =(\alpha_1,\ldots,\alpha_N)^\top\in\{0,1\}^N \) and \(\widehat x^{\bm\alpha}=\widehat x_1^{\alpha_1}\ldots \widehat x_N^{\alpha_N}\), and
let
\(c_{\bm \alpha}(\widehat v)\) 
denote the coefficient of the
monomial $\widehat x^{\bm\alpha}$ in \(\widehat v\in Q_1(\widehat K)\), i.e.,
\(
    \widehat v (\widehat x)=\sum_{\bm\alpha\in\{0,1\}^N}
    c_{\bm\alpha}(\widehat v)\widehat x^{\bm\alpha}.
\)
Let 
\(
    \mathcal A_N=\{\bm\alpha\in\{0,1\}^N:|\bm\alpha|\ge2\}.
\)

For any $\widehat v^\pm\in Q_1(\widehat K)$, motivated by  \eqref{jump_con_ref}, we impose the discrete interface conditions 
\begin{subequations}\label{dis_jump_Q1}
\begin{align}
\llbracket \widehat v^\pm \rrbracket |_{\widehat{\mathcal H}_K}&=0,\label{dis_jump_Q1_1}\\
\llbracket  \widehat{\mathcal{B}}_{h}^\pm \widehat \nabla \widehat v^\pm\cdot \widehat n_{h} \rrbracket (\widehat M_K)&=0,\label{dis_jump_Q1_2}\\
\llbracket  c_{\bm \alpha} (\widehat v^\pm) \rrbracket&=0
    \quad \forall \bm\alpha\in \mathcal A_N,\label{dis_jump_Q1_3}
\end{align}
\end{subequations}
where $\widehat M_K$ is a point in $\widehat K$, referred to as the flux-enforcement
point, whose selection will be specified later. 
We emphasize that the notation
\(
    \llbracket v^\pm\rrbracket:=v^+-v^-
\)
is used for scalar-, vector-, or matrix-valued functions, where \(v^\pm\) are
defined on the same domain. 
This is slightly different from the original definition of the jump
\(\llbracket u^\pm\rrbracket\), where \(u^\pm\) are defined on different
subdomains. Since no ambiguity arises, the same notation is used for simplicity.
Define, on the reference element,
\[
   \widehat V_K^{\rm IFE}(\widehat K)
    =
    \left\{
        \widehat v: \widehat v|_{\widehat K^\pm}=\widehat v^\pm,\, \widehat v^\pm\in Q_1(\widehat K),~ \widehat v^\pm\text{ satisfy   \eqref{dis_jump_Q1_1}--\eqref{dis_jump_Q1_3}}
    \right\}.
\]
The  local IFE space on the physical element \(K\) is
defined by
\[
   V^{\rm IFE}(K)
    =
    \left\{
       v : v= \widehat v\circ F_K^{-1}, ~
        \widehat v\in
        \widehat V_K^{\rm IFE}(\widehat K)
    \right\}.
\]

\begin{remark}
For the practical implementation of the numerical experiments in this
paper, we construct \(\widehat x_K\) and \(\widehat n_{h}\) from the intersection points of
\(\widehat\Gamma_K\) with the edges of \(\widehat K\). In two dimensions,
\(\widehat x_K\) is the midpoint of the two intersection points, and
\(\widehat n_{h}\) is a unit normal to the line segment joining them.
In three dimensions, following \cite{guo2020immersed}, we select three
intersection points that form a triangle satisfying a uniform maximum-angle
condition. We then take \(\widehat x_K\) as its centroid and
\(\widehat n_{h}\) as a unit normal to its plane.
Other choices are also possible. For example, one may choose
\(x_K\in\Gamma_K\) and set \(n_h=n(x_K)\). Alternatively, the
construction may also be based on an interpolated level-set function
\cite{ji2023immersed}.
\end{remark}


\begin{remark}
Condition~\eqref{dis_jump_Q1_1} enforces the continuity of the function values and tangential derivatives at the point \(\widehat{x}_K\), i.e.,
\[
 \llbracket
  \widehat v^\pm
    \rrbracket(\widehat x_K)=0,
    \quad
     \llbracket
  \widehat   \nabla 
  \widehat v^\pm \cdot  \widehat t_{i,h}
    \rrbracket(\widehat x_K)=0, \quad i=1,\ldots, N-1,
\]
where \(\{\widehat t_{i,h}\}_{i=1}^{N-1}\) is an orthonormal basis of the tangent space to \(\widehat{\mathcal H}_K\).
\end{remark}

\begin{remark}
Condition~\eqref{dis_jump_Q1_3} is equivalent to
\(
    \llbracket
        \widehat\partial^{\bm\alpha}\widehat v^\pm
    \rrbracket
    =0
\)
for all
\(
\bm\alpha\in \mathcal A_N,
\)
where
\(
    \widehat\partial^{\bm\alpha}
    =
    \frac{\partial^{|\bm\alpha|}}
    {\partial\widehat x_1^{\alpha_1}\cdots
     \partial\widehat x_N^{\alpha_N}}.
\)
This condition enforces the continuity of the higher-order derivatives, thereby providing the additional coupling conditions required for the unisolvence of the local IFE basis functions.
\end{remark}

\subsection{Derivation of explicit IFE basis functions} \label{subsec_explicit_ife_basis}
We first establish that every
\(\widehat\phi\in\widehat V_K^{\rm IFE}(\widehat K)\)
is uniquely determined by its nodal values.
A direct derivation would require solving a linear system for the \(2^{N+1}\) coefficients of the two pieces \(\widehat\phi^\pm\). 
We instead follow an augmented approach developed in \cite{ji2023immersed} to determine these coefficients.

Let \(\{\widehat\phi_i\}_{i=1}^{2^N}\) be the standard nodal basis functions of \(Q_1(\widehat K)\), and 
let \(\widehat\Pi \) be the standard nodal interpolation operator, i.e., 
\(
    \widehat\Pi\widehat v
    =
    \sum_{i=1}^{2^N}
    \widehat v(\widehat A_i)\widehat\phi_i,
\)
where \(\widehat A_i\) is the vertex associated with
\(\widehat\phi_i\).

For \(\widehat\phi \in \widehat V_K^{\rm IFE}(\widehat K)\),  once the augmented variable
\[
\alpha_K:=\llbracket \widehat \nabla \widehat\phi^\pm \cdot \widehat n_{h}\rrbracket(\widehat M_K)
\]
is known, the discrete jump conditions \eqref{dis_jump_Q1_1} and \eqref{dis_jump_Q1_3} imply that
\begin{equation}\label{deco_phi}
\widehat \phi =\widehat\Pi \widehat \phi+ (\widehat w-\widehat\Pi \widehat w)\alpha_K,
\end{equation}
where 
\begin{equation}\label{def_w}
\widehat w|_{\widehat K^\pm}= \widehat w^\pm,\quad  \widehat w^-=0,\quad \widehat w^+=\widehat \ell_K(\widehat x)=(\widehat x-\widehat x_K)\cdot\widehat n_{h}.
\end{equation}

It remains to determine  the augmented variable \(\alpha_K\) from the discrete flux condition
\eqref{dis_jump_Q1_2}. Substituting \eqref{deco_phi} into \eqref{dis_jump_Q1_2} gives
\[\llbracket \widehat n_{h}^\top \widehat{\mathcal{B}}_{h}^\pm \widehat \nabla (\widehat w^\pm-\widehat\Pi \widehat w) \rrbracket |_{\widehat M_K}\alpha_K=-\llbracket  \widehat n_{h}^\top\widehat{\mathcal{B}}_{h}^\pm \widehat \nabla \widehat\Pi \widehat \phi \rrbracket|_{\widehat M_K}.\]
Using \( \widehat \nabla \widehat w^+= \widehat n_{h}\) and \(\widehat w^-=0\), and
decomposing
\(\widehat\nabla\widehat\Pi\widehat w\)
into its normal and tangential components, we have
\[\llbracket
\widehat n_{h}^\top \widehat{\mathcal{B}}_{h}^\pm \widehat \nabla (\widehat w^\pm-\widehat\Pi \widehat w)
\rrbracket
= \widehat n_{h}^\top \widehat{\mathcal{B}}_{h}^+\widehat n_{h} - \llbracket \widehat n_{h}^\top   \widehat{\mathcal{B}}_{h}^\pm \widehat n_{h} \rrbracket \widehat \nabla \widehat\Pi \widehat w  \cdot \widehat n_{h} - \sum_{i=1}^{N-1} \llbracket \widehat n_{h}^\top   \widehat{\mathcal{B}}_{h}^\pm \widehat t_{i,h} \rrbracket (\nabla   \widehat\Pi\widehat w \cdot\widehat  t_{i,h}).\]
Consequently, the augmented variable satisfies
\begin{equation}\label{sol_alpha}
\mathcal D(\widehat M_K) \alpha_K = -\left.\left((\rho^+)^{-1} \llbracket  \widehat n_{h}^\top\widehat{\mathcal{B}}_{h}^\pm \widehat \nabla \widehat\Pi \widehat \phi \rrbracket \right)\right|_{ \widehat M_K}, 
\end{equation}
where
\[
\begin{aligned}
&\mathcal D(\widehat x)=1+\left(\frac{\rho^-}{\rho^+}-1\right) \widehat \nabla \widehat\Pi \widehat w  \cdot \widehat n_{h}-\sum_{i=1}^{N-1}\frac{1}{\rho^+}\llbracket \widehat n_{h}^\top   \widehat{\mathcal{B}}_{h}^\pm \widehat t_{i,h} \rrbracket (\nabla   \widehat\Pi\widehat w) \cdot\widehat  t_{i,h},
\\
&\rho^\pm(\widehat x ) = \widehat n_{h}^\top \widehat{\mathcal{B}}_{h}^\pm\widehat n_{h}>0.
\end{aligned}
\]

The remaining task is  to select an appropriate flux-enforcement point \(\widehat M_K\) such that
\(\mathcal D(\widehat M_K)\neq0\). 
The following geometric property provides a sufficient condition for this purpose.

\begin{lemma}\label{lem_point_req}
There exists a point
\(\widehat M_K\) such that
\begin{equation}\label{key_iden}
\widehat M_K\in\widehat K,\qquad 
    (\widehat\nabla\widehat\Pi\widehat w)(\widehat M_K)
    =
    \lambda\widehat n_h,
    \quad
    \lambda\in[0,1].
\end{equation}
\end{lemma}

\begin{proof}
The proof is given constructively in Section~\ref{sec_flux_point_selection}.
\end{proof}

We select the flux-enforcement points as in
Sections~\ref{subsec_quad_selection}
and~\ref{subsec_hex_selection}. Then, using \eqref{key_iden}  we have
 \begin{equation}\label{fenmu}
 \mathcal D(\widehat M_K) = 1+(\rho_K^-/\rho_K^+-1) \lambda\ge \min\{1, \rho_K^-/\rho_K^+\}\ge \min\{1, \beta_{\min}^-/\beta_{\max}^+\}>0,
 \end{equation}
 where $\rho_K^\pm=\rho^\pm(\widehat M_K)$.
 Therefore, \(\alpha_K\) is uniquely determined by
\eqref{sol_alpha}, and we obtain the following  result.

 \begin{lemma}\label{unisolve_IFE}
For any prescribed nodal values
\(\{c_i\}_{i=1}^{2^N}\), there exists a unique
\(\widehat\phi\in\widehat V_K^{\rm IFE}(\widehat K)\)
such that
\(\widehat\phi(\widehat A_i)=c_i\), \(i=1,\ldots,2^N\).
\end{lemma}

Define the local IFE basis functions
\(\{\widehat\phi_{i}^{\rm IFE}\}_{i=1}^{2^N}\) by
\begin{equation}\label{def_IFE_basis}
    \widehat\phi_{i}^{\rm IFE}
    \in
    \widehat V_K^{\rm IFE}(\widehat K),
    \qquad
    \widehat\phi_{i}^{\rm IFE}(\widehat A_j)
    =
    \delta_{ij},
    \qquad
    i,j=1,\ldots,2^N,
\end{equation}
where \(\delta_{ij}\) denotes the Kronecker delta.
Taking
\(\widehat\Pi\widehat\phi=\widehat\phi_i\)
in \eqref{deco_phi} gives 
\begin{equation}\label{phi_hat}
    \widehat\phi_{i}^{\rm IFE}
    =
    \widehat\phi_i
    +
    \alpha_{i}
    (
        \widehat w-\widehat\Pi\widehat w
    ),
    \qquad
    i=1,\ldots,2^N,
\end{equation}
where
\begin{equation}\label{phi_hat_fenmu}
    \alpha_{i}
    =
    -
   \left. \frac{
        (\rho^+)^{-1}
        \llbracket
            \widehat n_{h}^{\top}
            J_K^{-1} \widehat {\mathbb B}^\pm_{h}
       J_K^{-\top}
            \widehat\nabla\widehat\phi_i
        \rrbracket
    }{
       1+\left(\rho^-/\rho^+-1\right) \widehat \nabla \widehat\Pi \widehat w  \cdot \widehat n_{h}
    }\right|_{\widehat M_K}.
\end{equation}
The corresponding basis functions on the physical element \(K\) are
\[
    \phi_{i}^{\rm IFE}
    =
    \widehat\phi_{i}^{\rm IFE}\circ F_K^{-1},
    \qquad
    i=1,\ldots,2^N.
\]

\begin{remark}
If the centroid of \(\widehat\Gamma_{h,K}\) (the midpoint in two dimensions)
is chosen as the flux-enforcement point \(\widehat M_K\), then
\(\mathcal D(\widehat M_K)\) may vanish, leading to a loss of unisolvence
of the local IFE space. Two such counterexamples are provided in
Appendix~\ref{appendix}: one for scalar diffusion coefficients on a parallelogram and
the other for tensor-valued diffusion coefficients on a square.
\end{remark}

\section{Selection of flux-enforcement points}
\label{sec_flux_point_selection}
The purpose of this section is to select
a point \(\widehat M_K\) so that \eqref{key_iden} holds.
We first present three elementary results.
\begin{lemma}[Complementary configuration]\label{lem:complementarity}
By the complementary configuration, we mean the configuration obtained by
interchanging the roles of \(\widehat K^+\) and \(\widehat K^-\).
The corresponding  functions and unit normal vector
are defined by
\[
  \widehat w^{\,c}|_{\widehat K^-} = \widehat\ell_K^{\,c}(\widehat x)=-\widehat\ell_K(\widehat x),\quad
    \widehat w^{\,c}|_{\widehat K^+}=0, 
    \quad \widehat n_h^{\,c}=-\widehat n_h.
\]
Suppose that there exist
 \(\widehat M_K\in\widehat K\) and
\(\lambda^{c}\in[0,1]\) such that
\[
    \widehat\nabla
    \widehat\Pi\widehat w^{\,c}(\widehat M_K)
    =
    \lambda^{c}\widehat n_h^{\,c},
\]
then the same point \(\widehat M_K\) satisfies
\[
    \widehat\nabla
    \widehat\Pi\widehat w(\widehat M_K)
    =
    \lambda \widehat n_h,\quad \lambda=1-\lambda^{c}\in[0,1].
\]
\end{lemma}

\begin{proof}
By definition, we have
\(
    \widehat w
    =
    \widehat\ell_K+\widehat w^{\,c}.
\)
Since \(\widehat\ell_K\) is an affine function,  we have \(\widehat\Pi\widehat\ell_K=\widehat\ell_K\), and hence
\(
    \widehat\Pi\widehat w
    =
    \widehat\ell_K+\widehat\Pi\widehat w^{\,c}.
\)
Note that \( \nabla \widehat\ell_K =\widehat n_h\).
Taking the gradient at \(\widehat M_K\) proves the result.
\end{proof}

\begin{lemma}[Invariance under coordinate reflections and permutations]
\label{lem:coordinate_symmetry}
Let   \(\mathsf P\) be a permutation matrix and \( \mathsf S=\operatorname{diag}(s_1,\ldots,s_N)\),
\(s_i\in\{-1,1\}\).
Define
\(
\mathsf G=\mathsf P\mathsf S
\)
and
\(
    \widetilde x=\mathsf G\widehat x.
\)
Let
\(\widetilde K^\pm=\{ \widetilde x  : \widetilde x  =\mathsf G\widehat x, \widehat x\in \widehat K^\pm \}\).
We refer to the cutting configuration
\(\widetilde K^\pm\) as the
\(\mathsf G\)-transformed configuration of
\(\widehat K^\pm\).
For  the
\(\mathsf G\)-transformed configuration, 
\[
    \widetilde w^+=
    \widetilde\ell_K(\widetilde x)
    =
    \widehat\ell_K(\mathsf G^{-1}\widetilde x),
    \quad
    \widetilde w^-=0,\quad 
    \widetilde n_h
    =
   \mathsf G^{-\top}\widehat n_h.\]
Then \(|\widetilde n_h|=1\).  In addition, 
let
\(\widetilde\Pi\) denote the standard \(Q_1\) interpolation operator
in the \(\widetilde x\)-coordinates, and 
suppose that there exists \(\widetilde M_K\in\widehat K\) such that 
\(
    \nabla_{\widetilde x}
    \widetilde\Pi\widetilde w(\widetilde M_K)
    =
    \lambda\widetilde n_h
\) with
\(
    \lambda\in[0,1],
\)
then
\(
    \widehat M_K=\mathsf G^{-1}\widetilde M_K
\)
satisfies the requirement \eqref{key_iden}.
\end{lemma}

\begin{proof}
The relation \(|\widetilde n_h|=1\) is immediate since
\(|\widehat n_h|=1\) and \(\mathsf G\) is an orthogonal transformation.
It is easy to verify that 
\(
   \widehat\Pi\widehat w(\widehat x)
    =
    \widetilde\Pi\widetilde w(\widetilde x),
    \)
and  
\(\widehat x=\mathsf G^{-1} \widetilde x.
\)
By the chain rule and the fact \( \mathsf G^{-1} =\mathsf G^\top\), we have
\[
    \widehat\nabla\widehat\Pi\widehat w (\widehat M_K)
    =
    \mathsf G^\top
    \nabla_{\widetilde x} \widetilde\Pi\widetilde w(\widetilde M_K)=\lambda\mathsf G^\top\widetilde n_h
    =
    \lambda\mathsf G^\top\mathsf G^{-\top}\widehat n_h=
    \lambda\widehat n_h.
\]
\end{proof}

\begin{lemma}\label{lem:vertex_gradient}
Let \(    \bm\nu=(\nu_1,\ldots,\nu_N)^\top
    \in\widehat {\mathcal N}\) be a vertex of
\(\widehat K=[-1,1]^N\). 
Suppose that \(\widehat\ell_K\) is nonnegative at
\(\bm\nu\) and all its adjacent vertices, namely,
\[
    \widehat\ell_K(\bm\nu)\ge0
    \quad\text{and}\quad
    \widehat\ell_K(\bm\nu^{\langle i\rangle})\ge0,
    \qquad i=1,\ldots,N,
\]
where \(\bm\nu^{\langle i\rangle}\) denotes the vertex adjacent to
\(\bm\nu\) obtained by changing the sign of the \(i\)th component,
then
\(
    \widehat\nabla\widehat\Pi\widehat w(\bm\nu)
    =
    \widehat n_h.
\)
\end{lemma}

\begin{proof}
Recalling that
\(
\widehat w(\widehat x)|_{\widehat K^+}=
    \widehat\ell_K(\widehat x)
    =
    (\widehat x-\widehat x_K)\cdot\widehat n_h
\)
and
\(\widehat w|_{\widehat K^-}=0\),
we have
\(
\widehat\Pi\widehat w  = \widehat \ell_K 
\)
 on the edge connecting
\(
\bm\nu\) and \(\bm\nu^{\langle i\rangle}.
\)
Let \(\widehat n_h=(n_1,\ldots,n_{N})^\top\). 
Therefore, we have
\[
\bigl(
\partial_{\widehat x_i}\widehat\Pi\widehat w
\bigr)(\bm\nu)
=\bigl(
\partial_{\widehat x_i}\widehat \ell_K 
\bigr)(\bm\nu)
=
n_i,
\quad i=1,\ldots,N.
\]
This completes the proof.
\end{proof}

On \(\widehat K=[-1,1]^N\), the  standard \(Q_1\) nodal basis functions are
\begin{equation}\label{basis_ex}
    \widehat\phi_{\bm\nu}(\widehat x)
    =
    \prod_{i=1}^{N}
    \widehat\varphi_{\nu_i}(\widehat x_i),
    \quad
    \widehat\varphi_{\nu_i}(t)=\frac12(1+\nu_i t), 
\end{equation}
where  \(\bm\nu=(\nu_1,\ldots,\nu_N)^\top\in\widehat {\mathcal N}\) denotes a vertex of \(\widehat K\).

\subsection{Selection for quadrilateral elements}
\label{subsec_quad_selection}
In this subsection,
\(
    \widehat K=[-1,1]^2,
\)
and 
\(
    \widehat n_h=(n_1,n_2)^\top.
\)
After 
choosing
\(
    s_i=1
\)
if \(n_i\ge 0\),  
\(
    s_i=-1
\)
if \(n_i<0\),
and applying Lemma~\ref{lem:coordinate_symmetry}, it suffices to consider the case
\(
    n_i\geq 0,
\)
\(i=1,2\).
Then, the set \(\mathcal I_{h,K}^+\) defined in
\eqref{def_I_K_class} is upward closed in the sense that if
\(\bm\mu\geq\bm\nu\) componentwise, \(\bm\nu\in\mathcal I_{h,K}^+\), and
\(\bm\mu\in\widehat{\mathcal N}\), then
\(\bm\mu\in\mathcal I_{h,K}^+\).
Indeed, since \(n_i\geq0\) for all \(i\),
\(
    \widehat\ell_K(\bm{\mu})
    =
    \widehat\ell_K(\bm{\nu})
    +
    \widehat n_h\cdot
    (\bm{\mu}-\bm{\nu})
    \geq 0.
\)
Therefore, up to a permutation of coordinates, there are three nontrivial
canonical configurations. Ordered by decreasing cardinality of
\(\mathcal  I_{h,K}^+\), they are illustrated in
Fig.~\ref{fig:quad_flux_selection}.

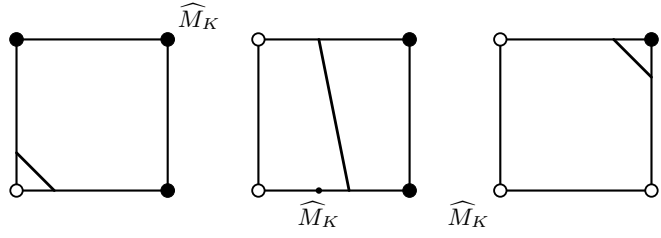
\begin{figure}[t]
\centering
\begin{tikzpicture}[
    scale=1,
    line cap=round,
    line join=round,
    every node/.style={font=\footnotesize}
]

\begin{scope}[xshift=0cm]

  \draw[thick] (-1,-1) rectangle (1,1);
  \draw[interface] (-1,-0.5)--(-0.5,-1);

  \Neg{(-1,-1)}
  \Pos{(1,-1)}
  \Pos{(-1,1)}
  \Pos{(1,1)}

  \node[above right] at (1,1) {$\widehat M_K$};

\end{scope}

\begin{scope}[xshift=3.2cm]

  \draw[thick] (-1,-1) rectangle (1,1);
  \draw[interface] (-0.2,1)--(0.2,-1);

  \Neg{(-1,-1)}
  \Neg{(-1,1)}
  \Pos{(1,-1)}
  \Pos{(1,1)}

  \fill (-0.2,-1) circle (1.2pt);
  \node[below] at (-0.2,-1) {$\widehat M_K$};

\end{scope}

\begin{scope}[xshift=6.4cm]

  \draw[thick] (-1,-1) rectangle (1,1);
  \draw[interface] (0.5,1)--(1,0.5);

  \Neg{(-1,-1)}
  \Neg{(1,-1)}
  \Neg{(-1,1)}
  \Pos{(1,1)}

  \node[below left] at (-1,-1) {$\widehat M_K$};

\end{scope}

\end{tikzpicture}
\caption{Canonical two-dimensional configurations, ordered by
decreasing number of nonnegative vertices.
Filled circles indicate nonnegative vertices.}
\label{fig:quad_flux_selection}
\end{figure}

\textbf{Case 2D-I: \(\lvert\mathcal  I_{h,K}^+\rvert=3 \).}
As shown in the first configuration in Fig.~\ref{fig:quad_flux_selection}, the vertex \((1,1)^\top\) and its two
adjacent vertices \((1,-1)^\top\) and \((-1,1)^\top\) belong to
\(\mathcal  I_{h,K}^+\). Lemma~\ref{lem:vertex_gradient} therefore gives
\[
    \widehat M_K=(1,1)^\top,
    \qquad
    \widehat\nabla
    \widehat\Pi\widehat w(\widehat M_K)
    =
    \widehat n_h.
\]

\textbf{Case 2D-II: \(\lvert\mathcal  I_{h,K}^+\rvert=2\).}
As shown in the second configuration in Fig.~\ref{fig:quad_flux_selection},
\begin{equation}\label{2d_2}
    \mathcal I_{h,K}^+
    =
    \{(1,-1)^\top,(1,1)^\top\}.
\end{equation}
Let
\(
    \delta:=\widehat\ell_K(1,-1)\geq0.
\)
Since   \((1,1)^\top\in\mathcal  I_{h,K}^+\) and  \((-1,-1)^\top\notin\mathcal  I_{h,K}^+\), by \eqref{def_w}, 
\[
\widehat\ell_K(1,1) =  \widehat\ell_K(1,-1)+2n_2=\delta+2n_2\ge 0,\qquad
\widehat\ell_K(-1,-1)
    =
    \delta-2n_1<0.
\]
Hence,
\(
    0\leq\delta<2n_1.
\)
By \eqref{basis_ex} and the definition of \(\widehat\Pi\), we have 
\[
    \widehat\Pi\widehat w(\widehat x)
    =\frac{\delta}{4}(1+\widehat x_1) (1-\widehat x_2)+\frac{\delta+2n_2}{4}(1+\widehat x_1) (1+\widehat x_2)
    =
    \frac{1+\widehat x_1}{2}
    \left[
        \delta+n_2(1+\widehat x_2)
    \right].
\]
Choose
\[
    \widehat M_K
    =
    \left(
        \frac{\delta}{n_1}-1,-1
    \right)^\top.
\]
 A direct calculation gives
\[
    \widehat\nabla
    \widehat\Pi\widehat w(\widehat M_K)
    =
    \frac{\delta}{2n_1}
(
        n_1, n_2
)^\top.
\]
Note that \(n_1>0\). Indeed, if \(n_1=0\), then
\(\widehat\ell_K\) would be independent of \(\widehat x_1\), and hence
\((-1,-1)^\top\) would also belong to \(\mathcal I^{\geq 0}\), which contradicts
\eqref{2d_2}.
The bound \(0\leq\delta<2n_1\) implies
\(
  \widehat M_K\in\widehat K,
  \)
  and
  \(\lambda=\delta/(2n_1)\in[0,1).
\)

\textbf{Case 2D-III: \(\lvert\mathcal  I_{h,K}^+\rvert=1\).}
Applying Lemma~\ref{lem:coordinate_symmetry} with \(
\mathsf G=\mathsf S
\)  and \(s_1=s_2=-1\), we get the \(\mathsf G\)-transformed configuration which is the complementary configuration of \textbf{Case~2D-I}. 
 By Lemma~\ref{lem:complementarity}, we can select the same point \((1,1)^\top\) as in \textbf{Case~2D-I} for the \(\mathsf G\)-transformed configuration. Finally, mapping the selected point back gives \(\widehat M_K=(-1,-1)^\top\). At this point, \(\widehat\nabla\widehat\Pi\widehat w(\widehat M_K) =0\) and \(\lambda=0\).

\subsection{Selection for hexahedral elements}
\label{subsec_hex_selection}

We now consider
\(
    \widehat K=[-1,1]^3
\)
and
\(
    \widehat n_h=(n_1,n_2,n_3)^\top.
\)
After 
choosing
\(
    s_i=1
\)
if \(n_i\ge 0\),  
\(
    s_i=-1
\)
if \(n_i<0\),
and applying Lemma~\ref{lem:coordinate_symmetry}, it suffices to consider 
\(
    n_i\geq0
\)
for
\( i=1,2,3
\).
Then, the set \(\mathcal  I_{h,K}^+\) defined in \eqref{def_I_K_class} is
upward closed. 
Up to a permutation of the coordinates, five
configurations with \(\lvert\mathcal  I_{h,K}^+\rvert\geq4\), ordered by decreasing
cardinality, are illustrated in Fig.~\ref{fig_3D}, 
 and the corresponding
nonnegative vertex sets \(\mathcal I_{h,K}^+\) are
\begin{equation}\label{eq:hex_canonical_configurations}
\begin{aligned}
\text{\rm(a)}\quad&
\widehat {\mathcal N}\setminus\{(-1,-1,-1)^\top\},
\\
\text{\rm(b)}\quad&
\widehat {\mathcal N}\setminus
\{(-1,-1,-1)^\top, (1,-1,-1)^\top\},
\\
\text{\rm(c)}\quad&
\widehat {\mathcal N}\setminus
\{(-1,-1,-1)^\top,
  (1,-1,-1)^\top,
  (-1,1,-1)^\top\},
\\
\text{\rm(d)}\quad&
\{(1,1,-1)^\top,(1,-1,1)^\top,(-1,1,1)^\top,(1,1,1)^\top\},
\\
\text{\rm(e)}\quad&
\{(1,-1,-1)^\top,(1,1,-1)^\top,(1,-1,1)^\top,(1,1,1)^\top\}.
\end{aligned}
\end{equation}

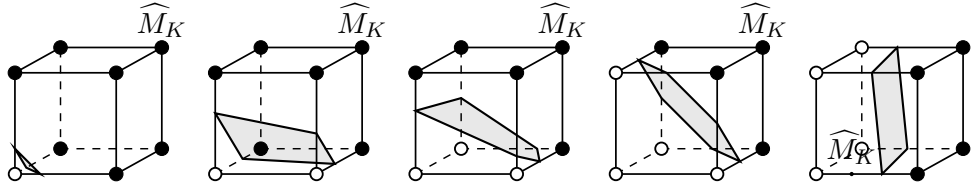
\begin{figure}[t]
\centering
\begin{tikzpicture}

\begin{scope}[xshift=0cm,
    x={(1.0cm,0cm)},
    y={(0.45cm,0.25cm)},
    z={(0cm,1.0cm)},
    scale=0.67]

  \filldraw[interfaceplane]
  (-1,-1,-0.5)--(-1,-0.5,-1)--(-0.5,-1,-1)--cycle;

  \CubeMinusOneOne
  \Neg{(-1,-1,-1)}
  \Pos{(1,-1,-1)}
  \Pos{(-1,1,-1)}
  \Pos{(-1,-1,1)}
  \Pos{(1,1,-1)}
  \Pos{(1,-1,1)}
  \Pos{(-1,1,1)}
  \Pos{(1,1,1)}

  \node[above] at (1,1,1) {$\widehat M_K$};
\end{scope}

\begin{scope}[xshift=2.65cm,
    x={(1.0cm,0cm)},
    y={(0.45cm,0.25cm)},
    z={(0cm,1.0cm)},
    scale=0.67]

  \filldraw[interfaceplane]
  (-1,0.2,-1)--(1,-0.2,-1)
  --(1,-1,-0.2)--(-1,-1,0.2)--cycle;

  \CubeMinusOneOne
  \Neg{(-1,-1,-1)}
  \Neg{(1,-1,-1)}
  \Pos{(-1,1,-1)}
  \Pos{(-1,-1,1)}
  \Pos{(1,1,-1)}
  \Pos{(1,-1,1)}
  \Pos{(-1,1,1)}
  \Pos{(1,1,1)}

  \node[above] at (1,1,1) {$\widehat M_K$};
\end{scope}

\begin{scope}[xshift=5.3cm,
    x={(1.0cm,0cm)},
    y={(0.45cm,0.25cm)},
    z={(0cm,1.0cm)},
    scale=0.67]

  \filldraw[interfaceplane]
  (-1,-1,0.25)
  --(1,-1,-0.65)
  --(1,0,-1)
  --(0.5,1,-1)
  --(-1,1,-0)
  --cycle;

  \CubeMinusOneOne
  \Neg{(-1,-1,-1)}
  \Neg{(1,-1,-1)}
  \Neg{(-1,1,-1)}
  \Pos{(-1,-1,1)}
  \Pos{(1,1,-1)}
  \Pos{(1,-1,1)}
  \Pos{(-1,1,1)}
  \Pos{(1,1,1)}

  \node[above] at (1,1,1) {$\widehat M_K$};
\end{scope}

\begin{scope}[xshift=7.95cm,
    x={(1.0cm,0cm)},
    y={(0.45cm,0.25cm)},
    z={(0cm,1.0cm)},
    scale=0.67]

  \filldraw[interfaceplane]
  (1,-1,0)--(1,0,-1)--(0,1,-1)
  --(-1,1,0)--(-1,0,1)--(0,-1,1)--cycle;

  \CubeMinusOneOne
  \Neg{(-1,-1,-1)}
  \Neg{(1,-1,-1)}
  \Neg{(-1,1,-1)}
  \Neg{(-1,-1,1)}
  \Pos{(1,1,-1)}
  \Pos{(1,-1,1)}
  \Pos{(-1,1,1)}
  \Pos{(1,1,1)}

  \node[above] at (1,1,1) {$\widehat M_K$};
\end{scope}

\begin{scope}[xshift=10.6cm,
    x={(1.0cm,0cm)},
    y={(0.45cm,0.25cm)},
    z={(0cm,1.0cm)},
    scale=0.67]

  \filldraw[interfaceplane]
  (-0.3,1,1)--(-0.1,1,-1)
  --(0.3,-1,-1)--(0.1,-1,1)--cycle;

  \CubeMinusOneOne
  \Neg{(-1,-1,-1)}
  \Pos{(1,-1,-1)}
  \Neg{(-1,1,-1)}
  \Neg{(-1,-1,1)}
  \Pos{(1,1,-1)}
  \Pos{(1,-1,1)}
  \Neg{(-1,1,1)}
  \Pos{(1,1,1)}

  \fill (-0.3,-1,-1) circle (1.2pt);
  \node[above] at (-0.3,-1,-1) {$\widehat M_K$};
\end{scope}

\end{tikzpicture}
\caption{Canonical three-dimensional configurations with at least four
nonnegative vertices, ordered by decreasing cardinality.
Filled circles indicate nonnegative vertices.}
\label{fig_3D}
\end{figure}

\textbf{Case 3D-I: \(\lvert\mathcal  I_{h,K}^+\rvert\geq5\).}
 As shown in the first three configurations in Fig.~\ref{fig_3D}, 
the vertex \((1,1,1)^\top\) and its three adjacent vertices
\(
    (1,1,-1)^\top,
    (1,-1,1)^\top,
    (-1,1,1)^\top
\)
belong to \(\mathcal  I_{h,K}^+\). Lemma~\ref{lem:vertex_gradient} therefore
gives
\[
    \widehat M_K=(1,1,1)^\top,
    \qquad
    \widehat\nabla
    \widehat\Pi\widehat w(\widehat M_K)
    =
    \widehat n_h.
\]

\textbf{Case 3D-II: \(\lvert\mathcal  I_{h,K}^+\rvert=4\),  Configuration~(d) in \eqref{eq:hex_canonical_configurations}.}
As shown in the  fourth  configuration in Fig.~\ref{fig_3D}, 
the vertex \((1,1,1)^\top\) and all its adjacent vertices belong to
\(\mathcal  I_{h,K}^+\). Hence Lemma~\ref{lem:vertex_gradient} gives
\(
    \widehat M_K=(1,1,1)^\top
\)
and
\( 
    \lambda=1.
\)

\textbf{Case 3D-III: \(\lvert\mathcal  I_{h,K}^+\rvert=4\),  Configuration~(e) in \eqref{eq:hex_canonical_configurations}.}
As shown in the  last  configuration in Fig.~\ref{fig_3D}, 
 the four nonnegative vertices form the face \(\widehat x_1=1\). In this case,
\(n_1>0\). Otherwise, \(\widehat\ell_K\) would be independent of
\(\widehat x_1\), and the corresponding vertices on the opposite face
would also belong to \(\mathcal  I_{h,K}^+\), which contradicts
\(\lvert\mathcal  I_{h,K}^+\rvert=4\).

Let
\(
    \delta:=\widehat\ell_K(1,-1,-1)\geq0.
\)
By \eqref{def_w} and the definition of \(\mathcal  I_{h,K}^+\), we have
\begin{equation}\label{3d_p_value}
\begin{aligned}
    &    \widehat\ell_K(1,1,-1)
    = \delta+2n_2\ge 0,\quad
        \widehat\ell_K(1,-1,1)
    =\delta+2n_3\ge 0,\\
   &        \widehat\ell_K(1,1,1)
    =\delta+2n_2+2n_3\ge 0,\quad
        \widehat\ell_K(-1,-1,-1)
    = \delta-2n_1<0.
 \end{aligned}
\end{equation}
Thus, we have the bound:
\(
    0\leq \delta<2n_1.
\)
By \eqref{basis_ex}, \eqref{3d_p_value} and the definition of \(\widehat\Pi\), we have 
\[
    \widehat\Pi\widehat w(\widehat x)
    =\sum_{\bm{\nu}\in \mathcal  I_{h,K}^+}\widehat\ell_K(\bm{\nu}) \widehat\phi_{\bm{\nu}}(\widehat x)=\frac{1+\widehat x_1}{2}
    \left(
        \delta+n_2(1+\widehat x_2)
         +n_3(1+\widehat x_3)
    \right).
\]
Choose
\[
    \widehat M_K
    =
    \left(
        \frac{\delta}{n_1}-1,-1,-1
    \right)^\top.
\]
 A direct calculation gives
\[
    \widehat\nabla
    \widehat\Pi\widehat w(\widehat M_K)
    =
    \frac{\delta}{2n_1} (n_1,n_2,n_3)^\top.
\]
The bound \(0\leq \delta<2n_1\) implies
 \(
 \widehat M_K\in\widehat K,
 \)
 and
 \(\lambda=\frac{\delta}{2n_1}\in[0,1).
 \)

\textbf{Case 3D-IV: \(\lvert\mathcal  I_{h,K}^+\rvert\leq3\).} 
Applying Lemma~\ref{lem:coordinate_symmetry} with \(\mathsf G=\mathsf S\) and \(s_i=-1, i=1,2,3\), 
we get the \(\mathsf G\)-transformed configuration which is the complementary configuration of \textbf{Case~3D-I}. 
 By Lemma~\ref{lem:complementarity}, we can select the same point \((1,1,1)^\top\) as in \textbf{Case~3D-I} for the \(\mathsf G\)-transformed configuration. Finally, mapping the selected point back gives \(\widehat M_K=(-1,-1,-1)^\top\). At this point, \(\widehat\nabla\widehat\Pi\widehat w(\widehat M_K)=0\) and \(\lambda=0\).

\subsection{Global IFE space}
The global  IFE space is defined by
\[
\begin{aligned}
V_h^{\rm IFE}=\{v_h:\;&v_h|_K\in V^{\rm IFE}(K)
        \,\forall K\in\mathcal K_h^{\Gamma},\, v_h|_K\in V(K)
        \,\forall K\in\mathcal K_h^{\rm non},\\
    &v_h\text{ is single-valued at every mesh vertex}\}.
\end{aligned}
\]
The subspace incorporating the homogeneous Dirichlet condition is
\[
V_{h,0}^{\rm IFE}
=\{v_h\in V_h^{\rm IFE}:v_h(A)=0
\text{ at every vertex }A \in\partial\Omega\}.
\]
Functions in \(V_h^{\rm IFE}\) are continuous across every face not
intersected by the interface, whereas across an interface face they are,
in general, discontinuous.

\section{Approximation Properties of IFE Spaces}\label{sec_approx}
Let \(d(x)|_{\Omega^\pm}=  \pm\operatorname{dist}(x,\Gamma)\) denote the signed distance function associated with \(\Gamma\).
For \(\delta>0\), define the tubular neighborhood
\[
    U(\Gamma,\delta)
    =
    \{x\in\Omega: |d(x)|<\delta\}.
\]
Since \(\Gamma\) is of class \(C^2\), there exists \(\delta_0>0\) such
that \(d\in C^2(U(\Gamma,\delta_0))\). On the interface \(\Gamma\), the unit normal satisfies
\(
    n=\nabla d.
\)
We extend the unit normal field
from \(\Gamma\) to \(U(\Gamma,\delta_0)\) by setting
\[
    n(x)=\nabla d(x),\qquad x\in U(\Gamma,\delta_0).
\]
Then \(n\in C^1(U(\Gamma,\delta_0))^N\) and \(|n|=1\) throughout this
neighborhood. 
We regard \(n_h\) as a piecewise constant vector
field on the interface elements.
Therefore, for each interface element \(K\subset U(\Gamma,\delta_0)\), it follows from
\eqref{est_x_K}  that
\begin{equation}\label{eq:normal_approximation_on_K}
    \|n_h-n\|_{L^\infty(K)}
    \leq
    |n_h-n(x_{\Gamma_K})|
    +
    \|n(x_{\Gamma_K})-n\|_{L^\infty(K)}\leq Ch_K.
\end{equation}

Throughout the paper, \(C\), \(C_1\), and \(C_2\) denote generic positive
constants that are independent of
the mesh size \(h\) and the position of the interface relative to the mesh, but may depend on the diffusion coefficient \(\mathbb{B}(x)\). 

We impose the following assumption on the mesh family.
\begin{assumption}
\label{assump:mesh_regularity}
The mesh family \(\{\mathcal{K}_h\}_{h>0}\) satisfies the following
conditions:
\begin{itemize}
  \item It is shape regular in the sense that 
  \begin{equation}\label{eq:isoparametric_regularity}
      \|J_K\|_{L^\infty(\widehat K)}
      \leq C h_K,
      \qquad
      \|J_K^{-1}\|_{L^\infty(\widehat K)}
      \leq C h_K^{-1}.
  \end{equation}

  \item It is asymptotically affine in the sense that
  \begin{equation}\label{eq:asymptotically_affine_derivatives}
      \|D^2F_K\|_{L^\infty(\widehat K)}
      \leq C h_K^2.
  \end{equation}
\end{itemize}
\end{assumption}

The asymptotically affine condition~\eqref{eq:asymptotically_affine_derivatives} is not specific to the proposed IFE method and is commonly used in the approximation analysis of mapped finite elements whose
reference shape-function spaces are not invariant under general
bilinear or trilinear mappings; see, e.g.,
\cite{92Simple,ArnoldBoffiFalk2002}.
This condition is automatically satisfied by mesh families obtained either
by regularly refining a fixed quadrilateral or hexahedral mesh or by
applying a fixed smooth mapping to Cartesian meshes.

In the present analysis, the asymptotically affine condition
\eqref{eq:asymptotically_affine_derivatives} is used in two places.
First, it enters the derivation of the IFE interpolation error estimates.
More precisely, it controls both the perturbation error arising from the
flux-enforcement condition \eqref{dis_jump_Q1_2} and the
higher-order coefficient term induced by the higher-order coupling condition
\eqref{dis_jump_Q1_3}; see the proof of
Lemma~\ref{lem:local_interpolation}. For the latter term, the chain rule
and \eqref{eq:asymptotically_affine_derivatives} give, with \(x=F_K(\widehat x)\),
\begin{equation}\label{ref_phy_H2}
\begin{aligned}
\bigl|
    \widehat{\nabla}^{\,2}\widehat v(\widehat x)
\bigr|
&\leq
C\lvert J_K(\widehat x)\rvert^2
\bigl|\nabla^2v(x)\bigr|
+
C\lvert D^2F_K(\widehat x)\rvert
\bigl|\nabla v(  x )\bigr| \\
&\leq
C h_K^2
\left(
    \bigl|\nabla^2v(  x )\bigr|
    +
    \bigl|\nabla v( x )\bigr|
\right),
\end{aligned}
\end{equation}
where the second term is generated by the non-affinity of
\(F_K\) and the asymptotically affine condition~\eqref{eq:asymptotically_affine_derivatives} ensures that this
term is also of order \(h_K^2\). 

Second, the asymptotically affine condition
\eqref{eq:asymptotically_affine_derivatives} yields the following
\(O(h_K)\) bound from \(\widehat\Gamma_K\) to
\(\widehat{\mathcal H}_K\), which is used in the consistency analysis;
see Lemma~\ref{lem:exact-interface-jump}.

\begin{lemma}\label{lem:interface_proximity}
Let \(K\in\mathcal K_h^\Gamma\) be an interface element. Recall the
definitions of \(\widehat{\mathcal H}_K\) and \(\widehat\ell_K\) in
\eqref{def_H_ref}. 
Then
\begin{equation}\label{eq:reference_interface_distance}
    \|\widehat\ell_K\|_{L^\infty(\widehat\Gamma_K)}
    \leq Ch_K.
\end{equation}
\end{lemma}
\begin{proof}
Let \(\widehat x\in\widehat\Gamma_K\). 
Then 
\(
    x=F_K(\widehat x)\in\Gamma_K.
\)
Let \(x_{\Gamma_K}\in\Gamma_K\) be the point appearing in
\eqref{est_x_K}. Taylor expansion
of \(d(x)\) about \(x_{\Gamma_K}\) gives
\[
    \bigl|
        (x-x_{\Gamma_K})\cdot n(x_{\Gamma_K})
    \bigr|
    \leq C|x-x_{\Gamma_K}|^2
    \leq Ch_K^2.
\]
Consequently, using \eqref{est_x_K}, we obtain
\[
    |(x-x_K)\cdot n_h|
    \leq
    |
        (x-x_{\Gamma_K})\cdot n(x_{\Gamma_K})
    |
    +
    |x-x_{\Gamma_K}|\,
    |n_h-n(x_{\Gamma_K})|
    +
    |x_{\Gamma_K}-x_K|
\leq Ch_K^2.
\]
Since \(x=F_K(\widehat x)\) and
\(x_K=F_K(\widehat x_K)\), a Taylor expansion of \(F_K\) about
\(\widehat x_K\), together with the identity
\[
    J_K(\widehat x_K)^\top  n_h
    =
    \frac{\widehat{n}_h}
    {\bigl|J_K(\widehat x_K)^{-\top}\widehat{n}_h\bigr|},
\]
gives
\[
\begin{aligned}
    (x-x_K)\cdot n_h
    &=
    \bigl[J_K(\widehat x_K)
    (\widehat x-\widehat x_K)\bigr]\cdot n_h
    +R_K(\widehat x)\cdot n_h
    \\
    &=
    \frac{
        (\widehat x-\widehat x_K)\cdot\widehat{n}_h
    }{
        \bigl|J_K(\widehat x_K)^{-\top}
        \widehat{n}_h\bigr|
    }
    +R_K(\widehat x)\cdot n_h .
\end{aligned}
\]
Here \(R_K(\widehat x)\) denotes the Taylor remainder. By the
asymptotically affine  condition
\eqref{eq:asymptotically_affine_derivatives} and the boundedness of
\(\widehat K\), it satisfies
\[
    |R_K(\widehat x)|
    \le
    C\|D^2F_K\|_{L^\infty(\widehat K)}
    |\widehat x-\widehat x_K|^2
    \le Ch_K^2,
    \qquad \widehat x\in\widehat K.
\]
Therefore, 
\[
\begin{aligned}
|\widehat\ell_K(\widehat x)|&=
|(\widehat x-\widehat x_K)\cdot\widehat n_h|\\
&\leq|J_K(\widehat x_K)^{-\top}\widehat n_h|
\left(|(x-x_K)\cdot n_h|+|R_K(\widehat x)\cdot n_h|\right)\leq Ch_K,
\end{aligned}
\]
which completes the proof.
\end{proof}

We also impose the following assumption.
\begin{assumption}\label{assump:mesh_size}
The mesh size \(h\) is sufficiently small that every interface element
\(K\in\mathcal K_h^\Gamma\) satisfies
\(
    K\subset U(\Gamma,\delta_0)
 \)
 and
 \(
    \|\widehat\ell_K\|_{L^\infty(\widehat\Gamma_K)}
    \leq 1/2.
\)
\end{assumption}

For \(v\in H^1(U(\Gamma,\delta_0))\) and $0<\delta\le\delta_0$, we shall use the standard strip estimates  (see, e.g., \cite{huang2002mortar,Li2010Optimal}):
\begin{equation}\label{eq:strip_estimates}
\begin{aligned}
&\|v\|_{L^2(U(\Gamma,\delta))}
\le C\sqrt{\delta}\| v\|_{H^1(U(\Gamma,\delta_0))},\\
&\|v\|_{L^2(U(\Gamma,\delta))}
\le C\delta\|\nabla v\|_{L^2(U(\Gamma,\delta))},
\quad\text{ if }v|_\Gamma=0.
\end{aligned}
\end{equation}


For $v^\pm\in H^2(\Omega^\pm)$, let $v_E^\pm$ denote extensions to
$U(\Gamma,\delta_0)$ satisfying (see, e.g., \cite{Gilbargbook})
\begin{equation}\label{eq:H2_extension}
    \|v_E^\pm\|_{H^2(U(\Gamma,\delta_0))}
    \le C\|v^\pm\|_{H^2(\Omega^\pm)}.
\end{equation}
We assume that \(\mathbb B^\pm\) admit \(C^1\) extensions
\(\mathbb B_E^\pm\) such that 
\begin{equation}\label{eq:coefficient_freezing}
    \|\mathbb B_E^\pm-\mathbb B_{h}^\pm\|_{L^\infty(K)}
    \le Ch_K.
\end{equation}

For any subdomain \(D\subset\Omega\), define the broken space
\begin{equation}\label{def_H2}
\widetilde H^2(D)
=
\left\{
    v\in H^2(D^+\cup D^-):
    \llbracket v^\pm\rrbracket=0,~
    \llbracket \mathbb B^\pm\nabla v^\pm\cdot n\rrbracket=0
    \ \text{on }\Gamma\cap D
\right\}.
\end{equation}
Problem~\eqref{p1} admits a unique solution
\(u\in\widetilde H^2(\Omega)\) satisfying (see, e.g., \cite{bramble1996finite,McLean})
\begin{equation}\label{regular}
    \|u\|_{H^2(\Omega^+\cup\Omega^-)}
    \leq
    C\|f\|_{L^2(\Omega)}.
\end{equation}

For \(v\in\widetilde H^2(\Omega)\), define the extended jumps in
\(U(\Gamma,\delta_0)\) by
\begin{equation}\label{def_extended_transmission_residuals}
    G_0(v)=\llbracket\nabla v_E^\pm\rrbracket,\quad
    G_\tau(v)=(\mathbb{I}-nn^\top)G_0(v),\quad
    G_n(v)=\llbracket
        \mathbb B_E^\pm\nabla v_E^\pm\cdot n
    \rrbracket,
\end{equation}
where \(\mathbb{I}\) is the \(N\times N\) identity matrix.
The interface conditions in \eqref{def_H2} imply
$G_\tau(v)|_\Gamma=0$ and $G_n(v)|_\Gamma=0$, whereas $G_0(v)$ generally has
a nonzero trace on $\Gamma$.
Taking \(\delta=h\) in the strip estimates
\eqref{eq:strip_estimates} and using
Assumption~\ref{assump:mesh_size}, we obtain
\begin{equation}\label{esti_strip}
    \sum_{K\in\mathcal K_h^\Gamma}
    \Big(
        \|G_\tau(v)\|_{L^2(K)}^2
        +
        \|G_n(v)\|_{L^2(K)}^2
    \Big)
    \leq
    Ch^2\|v\|_{H^2(\Omega^+\cup\Omega^-)}^2.
\end{equation}

\subsection{Bounds for IFE basis functions}
We establish uniform bounds for the IFE basis functions on \(K\in\mathcal K_h^\Gamma\).
\begin{lemma}\label{bound_IFE_basis}
Let \(\{\widehat\phi_i^{\rm IFE}\}_{i=1}^{2^N}\) be the IFE basis
functions defined in \eqref{def_IFE_basis}. Then
\[
    \left\|(\widehat\phi_i^{\rm IFE})^s\right\|_
    {W^{m,\infty}(\widehat K)}
    \leq C,
    \qquad
    i=1,\ldots,2^N,\quad
    m=0,1,\quad
    s\in\{+,-\}.
\]
\end{lemma}
\begin{proof}
The standard \(Q_1\) basis functions satisfy
\[
    \|\widehat\phi_i\|_{W^{m,\infty}(\widehat K)}
    \leq C,
    \qquad m=0,1.
\]
Therefore, by the explicit representation \eqref{phi_hat} of the IFE
basis functions, it suffices to establish the uniform bound
\(
    |\alpha_i|\leq C.
\)
By \eqref{frozen_B}, \eqref{B_h_trans}, \eqref{def_betaM}, and \eqref{eq:isoparametric_regularity}, we have
\[
    C_1h_K^{-2}
    \leq
    \rho^\pm
    :=
    \widehat n_h^\top
    \widehat{\mathcal B}_h^\pm
    \widehat n_h
    \leq
    C_2h_K^{-2}.
\]
 Similarly,
\[
    \left|
            \widehat n_h^\top
            \widehat{\mathcal B}_h^\pm
            \widehat\nabla\widehat\phi_i
    \right|
    \leq C|\widehat{\mathcal B}_h^\pm||\widehat\nabla\widehat\phi_i|\leq Ch_K^{-2}.
\]
Moreover, \eqref{fenmu} provides a uniform positive lower bound for the
denominator in the explicit expression \eqref{phi_hat_fenmu} for \(\alpha_i\). Hence,
\(
    |\alpha_i|\leq C.
\)
This completes the proof.
\end{proof}

\subsection{Local IFE interpolation error}
Recall that \(\widehat\Pi\) denotes the standard nodal \(Q_1\)
interpolation operator on the reference element \(\widehat K\), and let
\(\Pi_K\) denote the corresponding interpolation operator on the physical
element \(K\). For \(v\in\widetilde H^2(K)\), define its IFE interpolant
\(\Pi_K^{\rm IFE}v\) by
\[
    \Pi_K^{\rm IFE}v\in V^{\rm IFE}(K),
    \qquad
    (\Pi_K^{\rm IFE}v)(A_i)=v(A_i),
    \quad i=1,\ldots,2^N,
\]
where \(A_i=F_K(\widehat A_i)\).
The IFE interpolant on the reference
element  is defined by
\[
    \widehat\Pi_K^{\rm IFE}\widehat v
    \in\widehat V^{\rm IFE}_K(\widehat K),
    \qquad
    (\widehat\Pi_K^{\rm IFE}\widehat v)(\widehat A_i)
    =
    \widehat v(\widehat A_i),
    \quad i=1,\ldots,2^N.
\]
Let \(\widehat Z_K\) denote the space of piecewise-\(Q_1\) functions
whose values vanish at all vertices:
\[
    \widehat Z_K
    :=
    \left\{
        \widehat z
        \,:\,
        \widehat z|_{\widehat K^\pm}
        =
        \widehat z^\pm,~
        \widehat z^\pm\in Q_1(\widehat K),
      ~
        \widehat z(\widehat A_i)=0,
        ~ \widehat A_i\in \widehat {\mathcal N}
    \right\}.
\]
We emphasize that \(\widehat Z_K\) depends on the physical element \(K\),
since the subdomains
\(
    \widehat K^\pm=F_K^{-1}(K^\pm)
\)
depend on \(K\).

Define the  operator
\(
    \widehat{\mathscr R}_K:
    \widehat Z_K\longrightarrow\mathbb R^{2^N}
\)
by
\begin{equation}\label{def_residual_map}
\begin{aligned}
\widehat{\mathscr R}_K(\widehat z)
=
\Big(&
    \llbracket\widehat z^\pm\rrbracket(\widehat x_K),
    \big\{
        \llbracket
            \widehat\nabla\widehat z^\pm
            \cdot\widehat t_{i,h}
        \rrbracket(\widehat x_K)
    \big\}_{i=1}^{N-1},
    \\
    &h_K^2
    \llbracket
        \widehat{\mathcal B}_h^\pm
        \widehat\nabla\widehat z^\pm
        \cdot\widehat n_h
    \rrbracket(\widehat M_K),
    \big\{
        \llbracket
            c_{\bm\alpha}(\widehat z^\pm)
        \rrbracket
    \big\}_{\bm\alpha\in\mathcal A_N}
\Big).
\end{aligned}
\end{equation}
We then have the following uniform stability result.
\begin{lemma}\label{lem:auxiliary_functions}
The operator
\(
    \widehat{\mathscr R}_K
\)
is an isomorphism. Moreover, 
for every \(\widehat z\in\widehat Z_K\), 
\[
    \|\widehat z^s\|_{W^{m,\infty}(\widehat K)}
    \leq
    C\bigl|\widehat{\mathscr R}_K(\widehat z)\bigr|,
    \quad m=0,1,\quad s\in\{+,-\}.
\]
\end{lemma}
\begin{proof}
If \(\widehat z\in\ker\widehat{\mathscr R}_K\), then \(\widehat z\)
satisfies all the local IFE coupling conditions~\eqref{dis_jump_Q1_1}-\eqref{dis_jump_Q1_3} and vanishes at every
vertex. Hence, Lemma~\ref{unisolve_IFE} implies \(\widehat z=0\). Since
\(
    \dim\widehat Z_K=2^N
    =
    \dim\mathbb R^{2^N},
\)
the operator \(\widehat{\mathscr R}_K\) is an isomorphism.

To prove the uniform stability estimate, we let
\[
\bm r = 
\bigl(r_0,\{r_i\}_{i=1}^{N-1},r_n,
          \{r_{\bm\alpha}\}_{\bm\alpha\in\mathcal A_N}\bigr):= \widehat{\mathscr R}_K(\widehat z)
\]
and construct \(\widehat p^+\in Q_1(\widehat K)\) such that
\[
\begin{aligned}
  & \widehat p^+(\widehat x_K)=r_0,\quad
&&(\widehat \nabla \widehat p^+
            \cdot\widehat t_{i,h}
        )(\widehat x_K)=r_i,\quad i=1,\ldots,N-1,\\
    &h_K^2
   (
        \widehat{\mathcal B}_h^+
       \widehat \nabla \widehat p^+
        \cdot\widehat n_h
   )(\widehat M_K)=r_n,\quad
            &&c_{\bm\alpha}(\widehat p^+)=r_{\bm\alpha},\quad \bm\alpha\in\mathcal A_N.
\end{aligned}
\]
Since \(\widehat p^+\in Q_1(\widehat K)\), it can be written as 
\begin{equation}\label{def_p_plus}
  \widehat p^+(\widehat{\bm x})
  =a_0+\bm a\cdot\widehat{ x}+\widehat H(\widehat{ x}),
  \qquad
  \widehat H(\widehat{ x})
  :=\sum_{\bm\alpha\in\mathcal A_N}
        r_{\bm\alpha}\widehat{ x}^{\bm\alpha}.
\end{equation}
Obviously, the high-order term satisfies 
\begin{equation}\label{esti_H}
  \|\widehat H\|_{W^{m,\infty}(\widehat K)}
  \le C|\bm r|,
  \quad m=0,1,
\end{equation}
where  \(C\) is independent of \(h\) and the position of the interface.
Using the tangential derivative condition, we have
\begin{equation}\label{decom_a}
 |\bm a\cdot\widehat{ t}_{i,h}|
  =|r_i-\widehat\nabla\widehat H(\widehat x_K)
          \cdot\widehat{ t}_{i,h}|\leq C|\bm r|.
\end{equation}
Similarly, the normal derivative condition gives 
\begin{equation}\label{nor_deri}
 \bm a\cdot\widehat{n}_{h}
   =(\nabla \widehat p^+\cdot \widehat n_h)(\widehat M_K)-\widehat\nabla\widehat H(\widehat M_K)
          \cdot\widehat{n}_{h}.
 \end{equation}
Next, we estimate the first term on the right-hand side of the above equation.
Decomposing 
\(\widehat \nabla \widehat p^+\)
into its normal and tangential components yields 
\[
\widehat{\mathcal B}_h^+
       \widehat \nabla \widehat p^+
        \cdot\widehat n_h=\widehat n_h^\top\widehat{\mathcal B}_h^+\widehat n_h (\nabla \widehat p^+\cdot \widehat n_h
        )+
        \sum_{i=1}^{N-1}  \widehat n_{h}^\top   \widehat{\mathcal{B}}_{h}^+ \widehat t_{i,h}  (\nabla   \widehat p^+ \cdot\widehat  t_{i,h}).
\]
Therefore, at the point \(\widehat M_K\), we have
\[
\nabla \widehat p^+\cdot \widehat n_h=
(\widehat n_h^\top\widehat{\mathcal B}_h^+\widehat n_h)^{-1}
\Big(
h_K^{-2}r_n
        -\sum_{i=1}^{N-1}  \widehat n_{h}^\top   \widehat{\mathcal{B}}_{h}^+ \widehat t_{i,h}  
        (\bm a+\widehat\nabla\widehat H) \cdot\widehat  t_{i,h}\Big).
\]
 It follows from \eqref{frozen_B}, \eqref{B_h_trans}, \eqref{def_betaM}, \eqref{eq:isoparametric_regularity}, \eqref{esti_H} and \eqref{decom_a} that
 \[
 |(\nabla \widehat p^+\cdot \widehat n_h)(\widehat M_K)|\leq C|\bm r|,
 \]
 which, together with 
 \eqref{esti_H}, \eqref{decom_a} and \eqref{nor_deri}, yields 
\begin{equation}\label{esti_a}
\|\bm a\cdot\widehat{ x}\|_{W^{m,\infty}(\widehat K)}\leq C\sum_{i=1}^{N-1}| \bm a\cdot\widehat{ t}_{i,h}|+C|\bm a\cdot\widehat{n}_{h}|\leq C|\bm r|.
\end{equation}
The value condition at $\widehat x_K$ determines the constant term:
\begin{equation}\label{est_a_0}
  |a_0|=|r_0-\bm a\cdot\widehat x_K-\widehat H(\widehat x_K)|\leq C|\bm r|.
\end{equation}
Combining  \eqref{def_p_plus}, \eqref{decom_a}, \eqref{esti_a} and \eqref{est_a_0} yields 
\begin{equation}\label{eq:p_stability}
    \|\widehat p^+\|_{W^{m,\infty}(\widehat K)}
    \leq
    C|\bm r|
    \leq
    C\bigl|\widehat{\mathscr R}_K(\widehat z)\bigr|,
    \qquad m=0,1,
\end{equation}
where \(C\) is independent of \(h\) and the position of the interface.

Set \(\widehat p^-=0\) and  define \(\widehat p|_{\widehat K^\pm}=\widehat p^\pm\). Then, it is easy to verify that 
\[
    \widehat z
    =
    \widehat p-\widehat\Pi_K^{\rm IFE}\widehat p=\widehat p-
    \sum_{i=1}^{2^N}
    \widehat p(\widehat A_i)
    \widehat\phi_i^{\rm IFE}.
\]
Lemma~\ref{bound_IFE_basis} and \eqref{eq:p_stability} yield the desired result 
\[
    \|\widehat z^s\|_{W^{m,\infty}(\widehat K)}
    \leq
    \|\widehat p^+\|_{W^{m,\infty}(\widehat K)}
    +
    C\sum_{i=1}^{2^N}
    |\widehat p(\widehat A_i)|  
    \leq
    C\bigl|
        \widehat{\mathscr R}_K(\widehat z)
    \bigl|.
\]
\end{proof}

\begin{lemma}[Local interpolation estimate]\label{lem:local_interpolation}
Let $K\in\mathcal K_h^\Gamma$ and $v\in\widetilde H^2(\Omega)$.  For
$m=0,1$ and $s\in\{+,-\}$, the following estimate holds:
\begin{equation}\label{eq:local_interpolation}
\begin{aligned}
\left|v_E^s-(\Pi_K^{\rm IFE}v)^s\right|_{H^m(K)}
\le{}&C h_K^{2-m}
\sum_{s=\pm}\|v_E^s\|_{H^2(K)}\\
&+C h_K^{1-m}
\left(\|G_\tau(v)\|_{L^2(K)}
+\|G_n(v)\|_{L^2(K)}\right).
\end{aligned}
\end{equation}
\end{lemma}

\begin{proof}
Set \(
q^s:=\Pi_K v_E^s.
\)
Under Assumption~\ref{assump:mesh_regularity}, the following estimates hold:
\begin{equation}\label{standard_Q1_local_bounds}
\begin{aligned}
&h_K^j\|v_E^s-q^s\|_{H^j(K)}+ h_K^{N/2}\|v_E^s-q^s\|_{L^\infty(K)}\leq Ch_K^2\|v_E^s\|_{H^2(K)},\quad j=0,1,2,\\
&\|q^s\|_{H^1(K)}\leq C\|v_E^s\|_{H^2(K)}.
\end{aligned}
\end{equation}
Define a function \(q\) by 
\(
    q|_{K^\pm}=q^\pm.
\)
Set 
$\eta:=q -\Pi_K^{\rm IFE}v$. Then
\begin{equation}\label{pro_inter0}
\begin{aligned}
\left|v_E^s-(\Pi_K^{\rm IFE}v)^s\right|_{H^m(K)}&\leq \left|v_E^s-q^s\right|_{H^m(K)}+\left|q^s-(\Pi_K^{\rm IFE}v)^s\right|_{H^m(K)}\\
&\leq Ch_K^{2-m}\|v_E^s\|_{H^2(K)}+|\eta^s|_{H^m(K)}.
\end{aligned}
\end{equation}
By Assumption~\ref{ass:compatible_cut},
\(\Gamma_{h,K}\) and \(\Gamma_K\) induce the same partition of the
vertices of \(K\). Consequently, \(q\) and
\(\Pi_K^{\rm IFE}v\) have identical nodal values, and hence the
pullback \(\widehat\eta\) belongs to \(\widehat Z_K\).
Since
\(\widehat{\mathscr R}_K
(\widehat\Pi_K^{\rm IFE}\widehat v)=0\),  we have
\begin{equation}\label{eq:eta_residual_identity}
    \widehat{\mathscr R}_K(\widehat\eta)
    =\widehat{\mathscr R}_K(\widehat q-\widehat \Pi_K^{\rm IFE}\widehat v)=\widehat{\mathscr R}_K(\widehat q).
\end{equation}
Pulling \(\eta^s\) back to the reference element and using Lemma~\ref{lem:auxiliary_functions}, we obtain
\begin{equation}\label{pro_inter1}
|\eta^s|_{H^m(K)}\leq Ch_K^{N/2-m}\|\widehat \eta^s\|_{W^{m,\infty}(\widehat K)}\leq Ch_K^{N/2-m}\left|\widehat{\mathscr R}_K(\widehat q)\right|.
\end{equation}
We estimate the components of the residual \(\widehat{\mathscr R}_K(\widehat q)\) separately.

\textbf{Estimate of the value component.}
Since $\llbracket v_E^\pm\rrbracket(x_{\Gamma_K})=0$, we have
\[
|\llbracket \widehat q^\pm\rrbracket(\widehat x_K)|=|\llbracket  q^\pm\rrbracket( x_K)|
\le |\llbracket q^\pm\rrbracket(x_K)
      -\llbracket q^\pm\rrbracket(x_{\Gamma_K})|
+\sum_{s=\pm}|q^s(x_{\Gamma_K})-v_E^s(x_{\Gamma_K})|.
\]
Because $|x_K-x_{\Gamma_K}|\le Ch_K^2$, the first term is
bounded by
\[
 Ch_K^2\|\nabla\llbracket q^\pm\rrbracket\|_{L^\infty(K)}
 \le Ch_K^{2-N/2}\sum_{s=\pm}\|q^s\|_{H^1(K)}.
\]
Combining the above estimates and
\eqref{standard_Q1_local_bounds}
 yields
\begin{equation}\label{eq:value_residual_estimate}
 |\llbracket \widehat q \rrbracket(\widehat x_K)|
 \le Ch_K^{2-N/2}\sum_{s=\pm}\|v_E^s\|_{H^2(K)}.
\end{equation}

\textbf{Estimate of the tangential components.}
Let
\(
 t_{i,h}=
 J_K(\widehat x_K)\widehat t_{i,h}/
 |J_K(\widehat x_K)\widehat t_{i,h}|.
\)
Then 
\[
\left|\llbracket\widehat\nabla\widehat q^\pm
\cdot\widehat t_{i,h}\rrbracket(\widehat x_K)\right|
=
\left|J_K(\widehat x_K)\widehat t_{i,h}\right| \left| \llbracket\nabla q^\pm\rrbracket(x_K)\cdot t_{i,h}\right|\leq Ch_K
\left| \llbracket\nabla q^\pm\rrbracket(x_K)\cdot t_{i,h}\right|.
\]
Using the inverse estimate on the whole element, we have
\[
\begin{aligned}
\left|
\llbracket\widehat\nabla\widehat q^\pm
\cdot\widehat t_{i,h}\rrbracket(\widehat x_K)
\right|
&\le Ch_K^{1-N/2}
\|\llbracket\nabla q^\pm\rrbracket\cdot t_{i,h}\|_{L^2(K)}\\
&\leq Ch_K^{1-N/2}\left(\|\llbracket\nabla (q^\pm-v_E^\pm)\rrbracket\cdot t_{i,h}\|_{L^2(K)}+\|\llbracket \nabla v_E^\pm\rrbracket\cdot t_{i,h}\|_{L^2(K)}
\right).
\end{aligned}
\]
By \eqref{def_extended_transmission_residuals}, the relation $t_{i,h}\cdot n_{h}=0$, and  the estimate 
\eqref{eq:normal_approximation_on_K}, we further get
\[
\begin{aligned}
\left|\llbracket \nabla v_E^\pm\rrbracket\cdot t_{i,h}\right|&=
|G_0(v)\cdot t_{i,h}|
\le |G_\tau(v)|
 +|n^\top G_0(v) ((n-n_h)\cdot t_{i,h}+n_h\cdot t_{i,h})|\\
&\le |G_\tau(v)|
 +Ch_K\bigl(|\nabla v_E^+|+|\nabla v_E^-|\bigr).
\end{aligned}
\]
Combining the above estimates and 
\eqref{standard_Q1_local_bounds} yields
\begin{equation}\label{eq:tangential_residual_estimate}
\left|\llbracket\widehat\nabla\widehat q^\pm
\cdot\widehat t_{i,h}\rrbracket(\widehat x_K)\right|\leq
Ch_K^{2-N/2}\sum_{s=\pm}\|v_E^s\|_{H^2(K)}
+Ch_K^{1-N/2}\|G_\tau(v)\|_{L^2(K)}.
\end{equation}

\textbf{Estimate of the flux component.}
By  \eqref{B_h_trans} and the relation 
\[
n_{h}
=J_K(\widehat x_K)^{-\top}\widehat n_h/|J_K(\widehat x_K)^{-\top}\widehat n_h|,
\]
we have, for \(x=F_K(\widehat x), ~ \widehat x \in \widehat K\) and \(s\in\{+,-\}\),
\[
(\widehat{\mathcal B}_{h}^s
\widehat\nabla\widehat q^s\cdot\widehat n_h)(\widehat x)
=|J_K^{-\top}(\widehat x_K)\widehat n_h|
n_h^\top J_K(\widehat x_K)J_K^{-1}(\widehat x)\mathbb B_{h}^s\nabla q^s(x).
\]
It follows from  \eqref{eq:isoparametric_regularity} that
\[
 \left|\llbracket \widehat{\mathcal B}_{h}^\pm
 \widehat\nabla\widehat q^\pm\cdot\widehat n_{h}\rrbracket\right|
 \leq Ch_K^{-1}\left(\left|\llbracket n_h^\top\mathbb B_{h}^\pm\nabla q^\pm\rrbracket\right|+\left|n_h^\top(J_K(\widehat x_K)J_K^{-1}(\widehat x)-\mathbb{I})\llbracket \mathbb B_{h}^\pm\nabla q^\pm\rrbracket\right|\right).
\]
Using the asymptotically affine condition~\eqref{eq:asymptotically_affine_derivatives}, we have
\[
\begin{aligned}
|J_K(\widehat x_K)J_K^{-1}(\widehat x)-\mathbb{I}|
    &=
|(J_K(\widehat x_K)-J_K(\widehat x))J_K^{-1}(\widehat x)|
    \leq
|J_K(\widehat x_K)-J_K(\widehat x)||J_K^{-1}(\widehat x)|
\\
    &\leq
    Ch_K^{-1}
    \|D^2F_K\|_{L^\infty(\widehat K)}
    |\widehat x-\widehat x_K|
    \leq Ch_K.
\end{aligned}
\]
Therefore, 
\[
h_K^2\left|
\llbracket\widehat{\mathcal B}_{h}^\pm
\widehat\nabla\widehat q^\pm \cdot\widehat n_{h}
\rrbracket\right|\\
\le Ch_K^{1-N/2}
\|\llbracket
\mathbb B_{h}^\pm\nabla q^\pm\cdot n_{h}
\rrbracket\|_{L^2(K)}+Ch_K^{2-N/2}
\sum_{s=\pm}\|q^s\|_{H^1(K)}.
\]
Decompose the normal-flux jump in the first term as
\[
\llbracket
\mathbb B_{h}^\pm\nabla q^\pm\cdot n_{h}
\rrbracket\\
=\llbracket
\mathbb B_{h}^\pm\nabla(q^\pm-v_E^\pm)
\cdot n_{h}\rrbracket\\
+\llbracket
(\mathbb B_{h}^\pm-\mathbb B_E^\pm)
\nabla v_E^\pm\cdot n_{h}\rrbracket\\
+\llbracket
\mathbb B_E^\pm\nabla v_E^\pm\cdot(n_{h}-n)\rrbracket
+G_n(v),
\]
where \(G_n(v)\) is defined in \eqref{def_extended_transmission_residuals}.
Using \eqref{standard_Q1_local_bounds}, \eqref{eq:coefficient_freezing}, and
\eqref{eq:normal_approximation_on_K}, we have
\begin{equation}\label{eq:flux_residual_estimate}
h_K^2\left|
\llbracket\widehat{\mathcal B}_{h}^\pm
\widehat\nabla\widehat q^\pm \cdot\widehat n_{h}
\rrbracket(\widehat M_K)\right|
\leq
Ch_K^{2-N/2}\sum_{s=\pm}\|v_E^s\|_{H^2(K)}
+Ch_K^{1-N/2}\|G_n(v)\|_{L^2(K)}.
\end{equation}

\textbf{Estimate of the higher-order coefficient components.}
Let \(P_1(\widehat K)\) denote the space of polynomials of total degree at
most one on \(\widehat K\).
Since
\(
    c_{\bm\alpha}(\widehat q^s)
    =
    \widehat\partial^{\bm\alpha}
    \widehat q^s(\bm 0),
\)
the equivalence of norms on the finite-dimensional
quotient space \(Q_1(\widehat K)/P_1(\widehat K)\) yields, for every
\(\bm\alpha\in\{0,1\}^N\) with \(|\bm\alpha|\geq2\),
\[
|c_{\bm\alpha}(\widehat q^s)|
\leq \|\widehat\partial^{\bm\alpha}\widehat q^s\|_{L^\infty(\widehat K)}
\leq C |\widehat q^s|_{H^2(\widehat K)}.
\]
By \eqref{ref_phy_H2}, we obtain
\[
\lvert\widehat q^s\rvert_{H^2(\widehat K)}
\leq
C h_K^{2-\frac N2}
\left(
    \lvert q^s\rvert_{H^2(K)}
    +
    \lvert q^s\rvert_{H^1(K)}
\right)\leq
C h_K^{2-\frac N2}
\|q^s\|_{H^2(K)}.
\]
Consequently,
\begin{equation}\label{eq:mixed_residual_estimate}
\left|\llbracket  c_{\bm\alpha}(\widehat q^\pm)\rrbracket\right|
\le
Ch_K^{2-N/2}\sum_{s=\pm}\|q^s\|_{H^2(K)}
\le
Ch_K^{2-N/2}\sum_{s=\pm}\|v_E^s\|_{H^2(K)}.
\end{equation}

Combining \eqref{pro_inter0}--\eqref{eq:mixed_residual_estimate} with \eqref{def_residual_map} completes the proof.
\end{proof}

\subsection{Global interpolation estimates}
Define the global IFE interpolation operator \(\Pi_h^{\rm IFE}\)
elementwise by
\(\left.(\Pi_h^{\rm IFE}v)\right|_K=\Pi_Kv\) on non-interface elements
and \(\left.(\Pi_h^{\rm IFE}v)\right|_K=\Pi_K^{\rm IFE}v\) on interface
elements. By \eqref{esti_strip},  \eqref{eq:H2_extension}, and Lemma~\ref{lem:local_interpolation}, we have the following result.

\begin{theorem}\label{thm:optimal_interpolation}
For every \(v\in\widetilde H^2(\Omega)\), we have, for \(m=0,1\),
\[
    \sum_{K\in\mathcal K_h}
    \|v-\Pi_h^{\rm IFE}v\|_{H^m(K^+\cup K^-)}^2
    \leq
    Ch^{4-2m}
    \|v\|_{H^2(\Omega^+\cup\Omega^-)}^2.
\]
\end{theorem}

\section{IFE Formulation and Error Analysis}\label{sec_method}
Let \(\mathcal E_h\) denote the set of all \((N-1)\)-dimensional mesh
faces. A face \(e\in\mathcal E_h\) is called an interface face if
\(e\cap\Gamma\neq\emptyset\), and a non-interface face otherwise. The
corresponding sets are denoted by \(\mathcal E_h^\Gamma\) and
\(\mathcal E_h^{\rm non}\), respectively.
We use the standard notation for jumps and averages. For each interior face
\(e\), let \(K_1^e\) and \(K_2^e\) be the two elements
sharing \(e\), and fix a unit normal \(n_e(x)\) directed from \(K_1^e\) to
\(K_2^e\). We set
\(
    [v]_e:=v|_{K_1^e}-v|_{K_2^e},
\)
and
\(
    \{v\}_e:=\frac12
    \bigl(v|_{K_1^e}+v|_{K_2^e}\bigr).
\)

For each \(e\in\mathcal E_h^\Gamma\), let
\(
    \mathbf Q_e
    =
    \left\{
        \mathbf q \,:\,
        \mathbf q|_{K_i^e}=\nabla v_i, \,
        v_i\in V^{\rm IFE}(K_i^e), \, i=1,2
    \right\}.
\)
The local lifting operator
\(\mathbf r_e:L^2(e)\to\mathbf Q_e\) is defined by
\[
    \int_{K_1^e\cup K_2^e}
        \mathbb B\mathbf r_e(v)\cdot\mathbf q\,dx
    =
    \int_e
        v\,
        \bigl\{(\mathbb B\mathbf q)\cdot n_e\bigr\}_e\,ds
    \qquad
    \forall\,\mathbf q\in\mathbf Q_e.
\]

For \(v,w\in\widetilde H^2(\Omega)+V_h^{\rm IFE}\), define
\begin{subequations}\label{def_AAH}
\begin{align}
    A_h(v,w)
    &=
    a_h(v,w)+b_h(v,w)+s_h(v,w),\\
    a_h(v,w)
    &=
    \int_\Omega
        (\mathbb B\nabla_hv)\cdot\nabla_hw,
        \label{def_AAH_ah}\\
    b_h(v,w)
    &=
    -\sum_{e\in\mathcal E_h^\Gamma}
    \int_e
    \left(
        \bigl\{(\mathbb B\nabla_hv)\cdot n_e\bigr\}_e[w]_e
        +
        \bigl\{(\mathbb B\nabla_hw)\cdot n_e\bigr\}_e[v]_e
    \right),
        \label{def_AAH_bh}\\
    s_h(v,w)
    &=
    \eta_0
    \sum_{e\in\mathcal E_h^\Gamma}
    \int_{K_1^e\cup K_2^e}
        \mathbb B\mathbf r_e([v]_e)
        \cdot\mathbf r_e([w]_e),
        \label{def_AAH_sh}
\end{align}
\end{subequations}
where \(\eta_0\ge 12\), and \(\nabla_h\) is defined elementwise by
\(
    (\nabla_hv)|_{K^\pm}:=\nabla(v|_{K^\pm})
\)
for all
\( K\in\mathcal K_h.
\)
The immersed finite element method reads as follows: find
\(u_h\in V_{h,0}^{\rm IFE}\) such that
\begin{equation}\label{method_IFE}
    A_h(u_h,v_h)
    =
    \int_\Omega fv_h\,dx
    \qquad
    \forall\,v_h\in V_{h,0}^{\rm IFE}.
\end{equation}
\begin{remark}
With the lifting stabilization \eqref{def_AAH_sh}, coercivity holds
for every \(\eta_0\geq 12\).
Let $h_e$ denote the diameter of $e$.
One may use the penalty stabilization
\[
    \widetilde{s}_h(v,w)
    =
    \sum_{e\in\mathcal E_h^\Gamma}
    \frac{\eta_0}{h_e}
    \int_e [v]_e[w]_e.
\]
For this choice, however, \(\eta_0\) must be
 large enough  to ensure coercivity.
\end{remark}

\subsection{Error estimates}
Define the energy and augmented norms by 
\[
\begin{aligned}
\|v\|_h^2&=a_h(v,v),\\
\interleave v \interleave_h^2&=\|v\|_{h}^2+\sum_{e\in\mathcal{E}_h^\Gamma}\Big(h_e\|\{\mathbb{B}\nabla_h v\}_e\|^2_{L^2(e)}+h_e^{-1}\| [v]_e\|^2_{L^2(e)}\Big)+s_h(v,v).
\end{aligned}
\]
By the Cauchy--Schwarz inequality, \(A_h(\cdot,\cdot)\) satisfies
the continuity estimate
\begin{equation}\label{bound_A}
    |A_h(v,w)|
    \le C
    \interleave v\interleave_h
    \interleave w\interleave_h
    \qquad
    \forall\,v,w\in
    \widetilde H^2(\Omega)+V_h^{\rm IFE}.
\end{equation}
The coercivity of \(A_h(\cdot,\cdot)\) follows by adapting the proof of \cite[Lemma 5.1]{ji2023immersed} and observing that each interface element is counted at most six times in the sum over interface faces
(see the fourth configuration in Fig.~\ref{fig_3D}). We have, for \(\eta_0\ge 12\),
\begin{equation}\label{coer_A}
    A_h(v_h,v_h)
    \geq
    \frac12\|v_h\|_h^2
    \qquad
    \forall\,v_h\in V_{h,0}^{\rm IFE}.
\end{equation}

We next establish a uniform estimate for the polynomial extensions of the two pieces of a local IFE function.
\begin{lemma}
Let \(K\in\mathcal K_h^\Gamma\). For every \(v_h\in V^{\rm IFE}(K)\),
\begin{equation}\label{vh_pm_vh01}
    |v_h^s|_{H^m(K)}
    \le C
    |v_h|_{H^m(K^+\cup K^-)},
    \quad
    m=0,1,\quad s\in\{+,-\}.
\end{equation}
\end{lemma}
\begin{proof}

Under Assumption~\ref{assump:mesh_size}, an argument similar to that in
\cite[Lemma~4.3]{Ji2026JSC} shows that there exist
\(s_0\in\{+,-\}\) and \(\widehat{x}_0\in\widehat{K}^{s_0}\) such that
the ball \(B(\widehat{x}_0,r_0)\), centered at
\(\widehat{x}_0\) with radius \(r_0=1/4\), satisfies
\(
B(\widehat{x}_0,r_0)\subset\widehat{K}^{s_0}.
\)
By
\eqref{dis_jump_Q1_1} and \eqref{dis_jump_Q1_3},
\(\llbracket\widehat v_h^\pm\rrbracket\) is affine and vanishes on
\(\widehat{\mathcal H}_K\). Hence,
\(
\llbracket\widehat v_h^\pm\rrbracket
=
\widehat\ell_K\,
\llbracket\widehat\nabla\widehat v_h^\pm\rrbracket
\cdot\widehat{n}_h.
\)
Using \eqref{dis_jump_Q1_2}, we have
\[
\llbracket\widehat\nabla\widehat v_h^\pm\rrbracket
\cdot\widehat{n}_h
=
\frac{
\widehat{ n}_h^\top
(\widehat{\mathcal B}_h^--\widehat{\mathcal B}_h^+)
\widehat\nabla\widehat v_h^-}
{
\widehat{ n}_h^\top
\widehat{\mathcal B}_h^+
\widehat{ n}_h}(\widehat M_K)
=
\frac{
\widehat{ n}_h^\top
(\widehat{\mathcal B}_h^--\widehat{\mathcal B}_h^+)
\widehat\nabla\widehat v_h^+(\widehat M_K)}
{
\widehat{ n}_h^\top
\widehat{\mathcal B}_h^-
\widehat{ n}_h}(\widehat M_K).
\]
Therefore, 
\[
|\widehat {v}_h^+|_{H^m(\widehat K^-)}\leq |\widehat {v}_h^-|_{H^m(\widehat K^-)}+C| \nabla \widehat v_h^{s_0}(\widehat M_K)|.
\]
Let \(r_1=\operatorname{diam}(\widehat K)\). Define a ball \(B(\widehat x_0,r_1)\). 
Then \(B(\widehat x_0,r_0)\subset \widehat K \subset B(\widehat x_0,r_1)\). 
 An argument similar to that in \cite[Lemma 2.2]{2016High} gives
\[
\|\nabla \widehat v_h^{s_0}\|_{L^2(\widehat K)}\leq \|\nabla \widehat v_h^{s_0}\|_{L^2(B(\widehat x_0,r_1))}\leq C\|\nabla \widehat v_h^{s_0}\|_{L^2(B(\widehat x_0,r_0))},
\]
where the constant \(C\) is independent of the interface.
Therefore,
\[
| \nabla \widehat v_h^{s_0}(\widehat M_K)|\leq \|\nabla \widehat v_h^{s_0}\|_{L^\infty(\widehat K)}\leq C\|\nabla \widehat v_h^{s_0}\|_{L^2(\widehat K)}\leq C|\widehat v_h^{s_0}|_{H^m(B(\widehat x_0,r_0))}.
\]
Combining the above results, we have
\[
|\widehat {v}_h^+|_{H^m(\widehat K)}\leq C|\widehat {v}_h^-|_{H^m(\widehat K^-)}+C|\widehat v_h^{s_0}|_{H^m(\widehat K^{s_0})}\leq C|\widehat v_h|_{H^m(\widehat K^+\cup \widehat K^-)},
\]
which, together with the standard scaling
argument, proves \eqref{vh_pm_vh01} for \(s=+\). For \(s=-\), the proof is similar.
\end{proof}

Using the extension estimate \eqref{vh_pm_vh01} and arguing as in
\cite[Lemmas 4.4--4.5]{Ji2026JSC},  we get the following trace inequality on an interface element \(K\in\mathcal K_h^\Gamma\):
\[
\|\nabla_h v_h\|_{L^2(\partial K)}\leq C h_K^{-1/2}\|\nabla_h v_h\|_{L^2(K)}~\quad \forall v_h\in V^{\mathrm{IFE}}(K),
\]
and the stability estimate of the lifting operator $\mathbf{r}_e$:
\[
\|\mathbf{r}_e(v)\|_{L^2(K_1^e\cup K_2^e)}\leq C h_e^{-1/2}\|v\|_{L^2(e)}\quad \forall v\in L^2(e).
\]

The extension estimate \eqref{vh_pm_vh01} is also used in the following estimate.
\begin{lemma}\label{lem_jum_less}
Let $\mathcal{K}_h^e$ be the set of elements having $e$ as a face.  For all \(e\in\mathcal{E}_h^\Gamma\),
\[
\|[v_h]_e\|_{L^2(e)}^2\le C h_e\sum_{K\in\mathcal{K}_h^e}|v_h|^2_{H^1( K^+\cup K^-)} \quad  \forall v_h\in V_{h}^{\rm IFE}.
\]
\end{lemma}

\begin{proof}
Let \((\Pi_h v_h)|_K= \Pi_K (v_h|_K)\) for all \(K\in\mathcal{K}_h\). Then, on \(e\), we have
\begin{equation}\label{pro_edg_jump0}
[v_h]_e=[v_h-\Pi_h v_h]_e.
\end{equation}
We next estimate \(v_h-\Pi_Kv_h\) on \(K\in \mathcal{K}_h^e\). On the reference
element, \eqref{deco_phi} gives
\[
    \|\widehat v_h-\widehat\Pi\widehat v_h\|_{L^\infty(\widehat K)}
    =
     \|\llbracket
        \widehat\nabla\widehat v_h^\pm
        \cdot\widehat n_{h}
    \rrbracket
    \bigl(
        \widehat w-\widehat\Pi\widehat w
    \bigr)\|_{L^\infty(\widehat K)}\leq 
    C\|\llbracket
        \widehat\nabla\widehat v_h^\pm
        \cdot\widehat n_{h}
    \rrbracket \|_{L^2(\widehat K)}.
\]
Using the standard scaling relations for the isoparametric map, we have
\[
    \|v_h-\Pi_Kv_h\|_{L^2(e)}^2
    \leq     C h_K\|\llbracket
         \nabla  v_h^\pm
    \rrbracket \|^2_{L^2( K)}
    \leq Ch_K\sum_{s=\pm}|v_h^s|^2_{H^1(K)}.
\]
It follows from \eqref{vh_pm_vh01} that 
\[
    \|v_h-\Pi_Kv_h\|_{L^2(e)}^2\leq Ch_K  |v_h|^2_{H^1(K^+\cup K^-)},
\]
which, together with \eqref{pro_edg_jump0}, proves this lemma.
\end{proof}

Combining the above results, we conclude that
\(\|\cdot\|_h\) and \(\interleave\cdot\interleave_h\) are uniformly
equivalent on \(V_{h,0}^{\rm IFE}\):
\begin{equation}\label{norm_equ}
    C_1\|v_h\|_h
    \leq
    \interleave v_h\interleave_h
    \leq
    C_2\|v_h\|_h
    \qquad
    \forall\,v_h\in V_{h,0}^{\rm IFE}.
\end{equation}

Arguing as in \cite[Lemma 5.5]{ji2023immersed}, and using the local
interpolation estimate \eqref{eq:local_interpolation}, the strip
estimate \eqref{esti_strip}, the extension estimate
\eqref{eq:H2_extension}, and the standard trace inequality, we obtain
the following optimal interpolation error estimate in the augmented
norm:
\begin{equation}\label{aug_inter}
    \interleave
        v-\Pi_h^{\rm IFE}v
    \interleave_h
    \le
    Ch\,
    \|v\|_{H^2(\Omega^+\cup\Omega^-)}
    \qquad
    \forall v\in\widetilde H^2(\Omega).
\end{equation}

Since an IFE function generally has a nonzero jump across \(\Gamma\), this jump contributes to the consistency error. The following lemma estimates the jump and establishes a sharper bound for the IFE interpolant.
\begin{lemma}\label{lem:exact-interface-jump}
For every \(v_h\in V_h^{\rm IFE}\), the jump across the exact interface
satisfies
\begin{equation}\label{eq:discrete-exact-interface-jump}
    \left\|
        \llbracket v_h^\pm\rrbracket
    \right\|_{L^2(\Gamma)}
    \leq
    C h^{3/2}\|v_h\|_h.
\end{equation}
Moreover, for every \(v\in\widetilde H^2(\Omega)\), its IFE interpolant
\(v_I:=\Pi_h^{\rm IFE}v\) satisfies
\begin{equation}\label{interpolant-exact-interface-jump}
    \left\|
        \llbracket v_I^\pm\rrbracket
    \right\|_{L^2(\Gamma)}
    \leq
    C h^2
    \|v\|_{H^2(\Omega^+\cup\Omega^-)}.
\end{equation}
\end{lemma}
\begin{proof}
Consider \(K\in\mathcal K_h^\Gamma\). By
\eqref{dis_jump_Q1_1} and \eqref{dis_jump_Q1_3},
\(\llbracket\widehat v_h^\pm\rrbracket\) is affine and vanishes on
\(\widehat{\mathcal H}_K\). Hence,
\(
\llbracket\widehat v_h^\pm\rrbracket
=
\widehat\ell_K\,
\llbracket\widehat\nabla\widehat v_h^\pm\rrbracket
\cdot\widehat{n}_h.
\)
Using \eqref{eq:reference_interface_distance} and the standard scaling
argument, we obtain
\begin{equation}\label{eq:local-exact-interface-jump}
\begin{aligned}
\left\|\llbracket v_h^\pm\rrbracket\right\|_{L^2(\Gamma_K)}^2
&\leq
C h_K^{N-1}h_K^2
\sum_{s=\pm}
\|\widehat\nabla\widehat v_h^s\|_{L^2(\widehat K)}^2\leq
C h_K^3
\sum_{s=\pm}|v_h^s|_{H^1(K)}^2.
\end{aligned}
\end{equation}
Combining this estimate with \eqref{vh_pm_vh01} and summing over \(K\in\mathcal K_h^\Gamma\), we obtain \eqref{eq:discrete-exact-interface-jump}.

For \(v_I\), using the triangle inequality, the local
interpolation estimate \eqref{eq:local_interpolation}, and the strip and extension estimates
\eqref{eq:strip_estimates}, \eqref{esti_strip} and  \eqref{eq:H2_extension}, we have
\[
\begin{aligned}
\sum_{K\in\mathcal K_h^\Gamma}\sum_{s=\pm}
|v_I^s|_{H^1(K)}^2
&\leq
C\sum_{K\in\mathcal K_h^\Gamma}\sum_{s=\pm}
\left(
|v_E^s|_{H^1(K)}^2
+
|v_E^s-v_I^s|_{H^1(K)}^2
\right) \\
&\leq
C(h+h^2)\|v\|_{H^2(\Omega^+\cup\Omega^-)}^2
\leq
Ch\|v\|_{H^2(\Omega^+\cup\Omega^-)}^2.
\end{aligned}
\]
Applying \eqref{eq:local-exact-interface-jump} to \(v_h=v_I\) and summing
over the interface elements gives
\[
\left\|\llbracket v_I^\pm\rrbracket\right\|_{L^2(\Gamma)}^2
\leq
Ch^3
\sum_{K\in\mathcal K_h^\Gamma}\sum_{s=\pm}
|v_I^s|_{H^1(K)}^2
\leq
Ch^4\|v\|_{H^2(\Omega^+\cup\Omega^-)}^2.
\]
Taking square roots yields the desired result \eqref{interpolant-exact-interface-jump}.
\end{proof}

Using the above lemma, we can estimate  the consistency error of the proposed IFE method. By \eqref{method_IFE},
\eqref{p1}, and integration by parts, we obtain the following
error identity:
\begin{equation}\label{inde_consis}
    A_h(u-u_h,v_h)
    =
    -\int_\Gamma
    \mathbb B^-\nabla u^-\cdot\boldsymbol n\,
    \llbracket v_h^\pm\rrbracket\,ds
    \qquad
    \forall v_h\in V_{h,0}^{\rm IFE}.
\end{equation}
Applying the Cauchy--Schwarz
inequality, the global trace inequality, and
Lemma~\ref{lem:exact-interface-jump}, we obtain
\begin{equation}\label{consis_es}
    \left|A_h(u-u_h,v_h)\right|
    \leq
    Ch^{3/2}
    \|u\|_{H^2(\Omega^+\cup\Omega^-)}
    \|v_h\|_h
    \qquad
    \forall v_h\in V_{h,0}^{\rm IFE}.
\end{equation}

Combining the preceding results with the regularity estimate
\eqref{regular}, we obtain the following optimal error estimates.
\begin{theorem}
Let $u$ and $u_h$ solve \eqref{p1} and \eqref{method_IFE},
respectively.  Then
\[
   \|u-u_h\|_{L^2(\Omega)} + h\interleave u-u_h\interleave_h
    \le Ch^2\|u\|_{H^2(\Omega^+\cup\Omega^-)}.
\]
\end{theorem}
\begin{proof}
In view of \eqref{coer_A}, \eqref{bound_A}, \eqref{norm_equ},
\eqref{aug_inter}, and \eqref{consis_es}, the proof of the following
energy-norm estimate is standard:
\begin{equation}\label{eq:energy-error-final}
 \| u-u_h\|_h+\interleave u-u_h\interleave_h
  \leq Ch\|u\|_{H^2(\Omega^+\cup\Omega^-)}.
\end{equation}

Let \(e=u-u_h\), and let \(z\in\widetilde H^2(\Omega)\) solve the adjoint interface problem associated with \eqref{p1} with right-hand side \(e\). By \eqref{regular}, the regularity
estimate gives
\begin{equation}\label{eq:dual-regularity}
 \|z\|_{H^2(\Omega^+\cup\Omega^-)}\le C\|e\|_{L^2(\Omega)}.
\end{equation}
Set \(z_I=\Pi_h^{\rm IFE}z\).
By integration by parts and \eqref{inde_consis}, we have
\[
\begin{aligned}
\|e\|_{L^2(\Omega)}^2
&=A_h(e,z)+\int_\Gamma \mathbb B^-\nabla z^-\cdot n
       \llbracket e^\pm\rrbracket\,ds\\
&=A_h(e,z-z_I)
 -\int_\Gamma \mathbb B^-\nabla u^-\cdot n\llbracket z_I^\pm\rrbracket\,ds
 -\int_\Gamma \mathbb B^-\nabla z^-\cdot n\llbracket u_h^\pm\rrbracket\,ds .
\end{aligned}
\]
Moreover, Lemma~\ref{lem:exact-interface-jump} and the energy estimate
\eqref{eq:energy-error-final} imply
\[
\begin{aligned}
\|\llbracket u_h^\pm\rrbracket\|_{L^2(\Gamma)}
&\le
 \|\llbracket u_I^\pm\rrbracket\|_{L^2(\Gamma)}
 +\|\llbracket (u_h-u_I)^\pm\rrbracket\|_{L^2(\Gamma)}\\
&\le Ch^2\|u\|_{H^2(\Omega^+\cup\Omega^-)}
 +Ch^{3/2}\|u_h-u+u-u_I\|_h
 \le Ch^2\|u\|_{H^2(\Omega^+\cup\Omega^-)}.
\end{aligned}
\]
Combining the above results with continuity, the interpolation estimates, 
the trace inequality, and Lemma~\ref{lem:exact-interface-jump} yields 
\[
\|e\|_{L^2(\Omega)}^2 \le Ch^2\|u\|_{H^2(\Omega^+\cup\Omega^-)}\|z\|_{H^2(\Omega^+\cup\Omega^-)}.
\]
Using \eqref{eq:dual-regularity} and dividing by
\(\|e\|_{L^2(\Omega)}\) proves the \(L^2\)-estimate. 
\end{proof}

\section{Numerical Experiments}\label{sec_num}
We present numerical experiments to validate the theoretical results. 
Let \(\varphi_\Gamma\) be a level set function such that
\[
\Gamma=\bigl\{  x\in\Omega:\varphi_\Gamma(  x)=0\bigr\},
\qquad
\Omega^+=\bigl\{  x\in\Omega:\varphi_\Gamma(  x)\geq0\bigr\}.
\]
Given  \(\varphi_\Gamma\) and  \(\mathbb B^\pm\), the exact solution is chosen as
\[
u^+
=
e^{-\varphi_\Gamma}\varphi_\Gamma
(\nabla\varphi_\Gamma)^\top\mathbb B^-\nabla\varphi_\Gamma,
\qquad
u^-
=
e^{-\varphi_\Gamma}\varphi_\Gamma
(\nabla\varphi_\Gamma)^\top\mathbb B^+\nabla\varphi_\Gamma.
\]
By construction, \(u\) satisfies the interface conditions
\eqref{p1.2}--\eqref{p1.3}. Since the manufactured solution generally does
not vanish on \(\partial\Omega\), we replace \eqref{p1.4} by \(u=g\) on \(\partial\Omega\) and
treat the resulting nonhomogeneous boundary condition in the standard
finite element manner. The source term \(f\) and the Dirichlet data \(g\)
are determined from the exact solution.

Starting from an initial mesh \(\mathcal K_{h_0}\), we generate
\(\mathcal K_{h_l}\), \(l\geq1\), by successive uniform refinement.
Let \(u_{h_l}\) denote the IFE solution on \(\mathcal K_{h_l}\).
We measure the
relative \(L^2\) and broken \(H^1\)-seminorm errors by
\[E_{0,l}
=
\frac{\|u-u_{h_l}\|_{L^2(\Omega)}}{\|u\|_{L^2(\Omega)}},
\qquad
E_{1,l}
=
\frac{\|\nabla_h (u-u_{h_l})\|_{L^2(\Omega)}}
     {|u|_{H^1(\Omega^+\cup\Omega^-)} }.
\]
We test the following two examples. Their computational domains, interfaces, and initial meshes are shown in
Fig.~\ref{fig:domains_meshes}.

\noindent
\textbf{Example 1 (Two-dimensional problem).}
We choose
\[\begin{aligned}
&\varphi_{\Gamma}(x_1,x_2)
=
\left(3(x_1^2+x_2^2)-x_1\right)^2
-x_1^2-x_2^2+0.02,\\
&\mathbb{B}^+=
\begin{pmatrix}
2 & -6\\
-6 & 19
\end{pmatrix},\quad
\mathbb{B}^-=
\begin{pmatrix}
1 & -4\\
-4 & 17
\end{pmatrix}.
\end{aligned}
\]

\noindent
\textbf{Example 2 (Three-dimensional problem).}
We choose
\[\begin{aligned}
&\varphi_{\Gamma}(x_1,x_2,x_3)
=x^2+y^2+z^2-0.5^2,\\
&\mathbb{B}^+(x_1,x_2,x_3)=
\begin{pmatrix}
4x_1^{2}+6 & \sin(x_1+x_2) & x_1x_2 \\
\sin(x_1+x_2) & 2x_3^{2}+3 & 0.5\sin(x_1) \\
x_1x_2 & 0.5\sin(x_1) & \cos^{2}(x_1x_2+x_3)+5
\end{pmatrix},\\
&\mathbb{B}^{-}(x_1,x_2,x_3)=
\begin{pmatrix}
\cos^{2}(x_1+x_2)+3 & x_3 & 0.2\sin(x_3-x_1) \\
x_3 & x_3^{2}+5 & x_2 \\
0.2\sin(x_3-x_1) & x_2 & \sin^{2}(x_3)+2
\end{pmatrix}.
\end{aligned}
\]
\begin{figure}[htbp]
    \centering
    \includegraphics[height=0.25\linewidth]{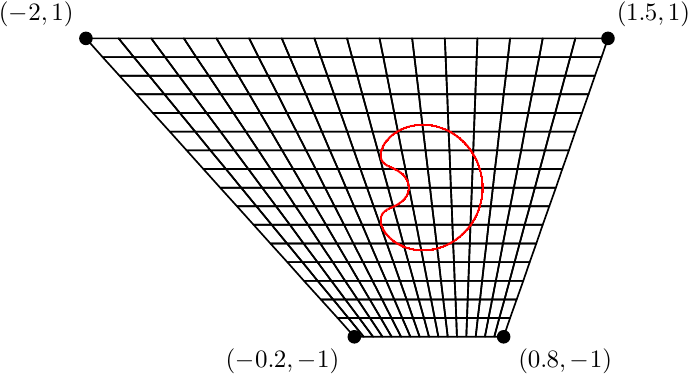}
     \includegraphics[height=0.25\linewidth]{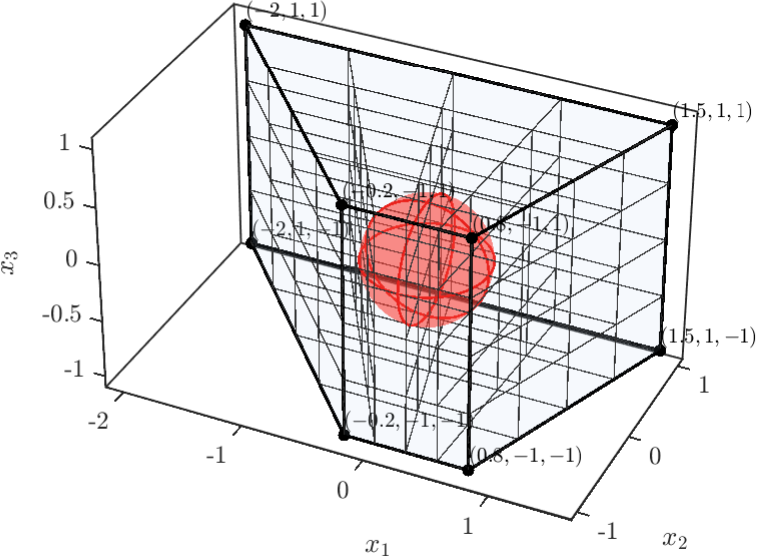}
    \caption{Computational domains, interfaces, and initial meshes for Example~1 (left) and Example~2 (right).}\label{fig:domains_meshes}
\end{figure}

The errors and convergence rates for both examples are reported in Table~\ref{tab:error}.
The results exhibit second-order convergence in the \(L^2\) norm and
first-order convergence in the broken \(H^1\) seminorm for both examples,
in agreement with the theoretical analysis.

\begin{table}[htbp]
\centering
\caption{Relative \(L^2\) and broken \(H^1\)-seminorm errors and convergence
rates for Examples~1 and~2.}
\label{tab:error}
\setlength{\tabcolsep}{3pt}
\resizebox{0.85\textwidth}{!}{%
\begin{tabular}{c|cc|cc|cc|cc}
\toprule
& \multicolumn{4}{c|}{Example~1}
& \multicolumn{4}{c}{Example~2} \\
\cline{2-9}
\(l\)
& \(E_{0,l}\) & rate
& \(E_{1,l}\) & rate
& \(E_{0,l}\) & rate
& \(E_{1,l}\) & rate \\
\hline
0
& \(8.8168\times10^{-2}\) & --
& \(3.7722\times10^{-1}\) & --
& $1.3157\times10^{-1}$ & --
& $3.9511\times10^{-1}$ & --\\
1
& \(2.4160\times10^{-2}\) & 1.87
& \(1.9321\times10^{-1}\) & 0.97
& $3.6878\times10^{-2}$ & 1.84
& $2.1491\times10^{-1}$ & 0.88\\
2
& \(6.3967\times10^{-3}\) & 1.92
& \(9.7025\times10^{-2}\) & 0.99
& $9.2086\times10^{-3}$ & 2.00
& $1.0771\times10^{-1}$ & 1.00\\
3
& \(1.6475\times10^{-3}\) & 1.96
& \(4.8533\times10^{-2}\) & 1.00
& $2.2992\times10^{-3}$ & 2.00
& $5.3927\times10^{-2}$ & 1.00\\
4
& \(4.1635\times10^{-4}\) & 1.98
& \(2.4265\times10^{-2}\) & 1.00
& $5.7263\times10^{-4}$ & 2.01
& $2.6900\times10^{-2}$ & 1.00\\
5
& \(1.0441\times10^{-4}\) & 2.00
& \(1.2132\times10^{-2}\) & 1.00
&  & 
&  &   \\
6
& \(2.6125\times10^{-5}\) & 2.00
& \(6.0660\times10^{-3}\) & 1.00
&   &  
&   &   \\
\bottomrule
\end{tabular}%
}
\end{table}

\appendix 
\section{Failure of Unisolvence for Centroid-Based Flux Enforcement}
\label{appendix}
Direct symbolic computation shows that, in each of the following examples, the linear system
determining the coefficients of the local IFE functions is singular;
hence, unisolvence fails.

\textbf{A scalar-coefficient counterexample on a mildly skewed parallelogram}
\begingroup
\[
\begin{gathered}
    \mathbb B^\pm=\beta^\pm\mathbb I,\qquad
    \beta^+=61,\quad \beta^-=1,\qquad
    A_3,A_4\in\Omega^+,                                                   \\[-2pt]
    A_1=\left(-\frac{53}{25},-\frac{21}{25}\right),\quad
    A_2=\left(\frac{4}{25},-\frac{22}{25}\right),\quad
    A_3=\left(\frac{53}{25},\frac{21}{25}\right),\quad
    A_4=\left(-\frac{4}{25},\frac{22}{25}\right),                         \\[-2pt]
    D=\left(-\frac{109}{60},-\frac{241}{420}\right),\qquad
    E=\left(\frac{11}{60},-\frac{361}{420}\right),\quad
        \widehat{\mathcal H}_K= \widehat D\widehat E,\quad
    \widehat M{_K}=\frac{\widehat D+\widehat E}{2}.
\end{gathered}
\]
\endgroup

\textbf{A tensor-coefficient counterexample on a square}
\[
\begin{gathered}
 K=[-1,1]^2,\quad
    \mathbb B^+
    =
    \begin{pmatrix}
        5 & -3\\
        -3 & 2
    \end{pmatrix},
    \quad 
    \mathbb B^-
    =
    \begin{pmatrix}
        5 & -\dfrac{24}{11}\\[2pt]
        -\dfrac{24}{11} & 1
    \end{pmatrix},\quad A_1\in\Omega^-,\\
    D=\widehat D=\left(-1,-\frac{1}{2}\right),\quad
    E=\widehat E=\left(0,-1\right),\quad 
        \widehat{\mathcal H}_K= \widehat D\widehat E,\quad
    \widehat M{_K}=\frac{\widehat D+\widehat E}{2}.
\end{gathered}
\]

\bibliographystyle{siamplain}
\bibliography{refer}
\end{document}